\documentclass{article}
\usepackage[utf8]{inputenc}

\usepackage{amsmath, mathtools, amsfonts, amssymb, amsthm, bm, float, subcaption}
\usepackage[shortlabels]{enumitem}
\usepackage[linesnumbered, ruled]{algorithm2e}
\usepackage{color}
\usepackage[title]{appendix}

\usepackage[margin=1.0in]{geometry}

\newcommand\bbR{\ensuremath{\mathbb{R}}} 
\newcommand\bbN{\ensuremath{\mathbb{N}}} 
\newcommand\bbE{\ensuremath{\mathbb{E}}} 
\newcommand\bbP{\ensuremath{\mathbb{P}}} 
\newcommand\bbS{\ensuremath{\mathbb{S}}} 
\newcommand\bbB{\ensuremath{\mathbb{B}}} 
\DeclareMathOperator*{\tr}{tr} 
\DeclareMathOperator*{\var}{Var} 

\newcommand{\ep}{\epsilon} 

\DeclarePairedDelimiter\abs{\lvert}{\rvert}%
\DeclarePairedDelimiterX{\norm}[1]{\lVert}{\rVert}{#1}

\DeclarePairedDelimiter{\brac}{\langle}{\rangle}

\newcounter{relctr}[section] 

\newcommand\labelrel[2]{%
  \begingroup
    \refstepcounter{relctr}%
    \stackrel{\textnormal{(\roman{relctr})}}{\mathstrut{#1}}%
    \originallabel{#2}%
  \endgroup
}
\AtBeginDocument{\let\originallabel\label}

\newtheorem{theorem}{Theorem}
\newtheorem{corollary}{Corollary}
\newtheorem{lemma}{Lemma}
\newtheorem{proposition}{Proposition}
\newtheorem{assumption}{Assumption}

\theoremstyle{definition}

\theoremstyle{remark}

\allowdisplaybreaks

\usepackage{thmtools}
\usepackage{thm-restate}
\usepackage{cleveref}

\title{\LARGE \bf
LQR with Tracking: A Zeroth-order  Approach and Its Global Convergence
}

\author{Zhaolin Ren${^*}$, Aoxiao Zhong${^*}$, Na Li${^*
}$ \thanks{Zhaolin Ren, Aoxiao Zhong and Na Li are with the School of Engineering and Applied Sciences at Harvard University. Emails: zhaolinren@g.harvard.edu, aoxiaozhong@fas.harvard.edu, nali@seas.harvard.edu.}%
\thanks{The work is funded by NSF CAREER ECCS-1553407, AFOSR YIP: FA9550-18-1-0150
ONR YIP: N00014-19-1-2217}
}

\begin{document}

\maketitle
\thispagestyle{empty}
\pagestyle{plain}

\begin{abstract}

There has been substantial recent progress on the theoretical understanding of model-free approaches to Linear Quadratic Regulator (LQR) problems. Much attention has been devoted to the special case when the goal is to drive the state close to a zero target. In this work, we consider the general case where the target is allowed to be arbitrary, which we refer to as the LQR tracking problem. We study the optimization landscape of this problem, and show that similar to the zero-target LQR problem, the LQR tracking problem also satisfies gradient dominance and local smoothness properties. This allows us to develop a zeroth-order  policy gradient algorithm that achieves global convergence. We support our arguments with numerical simulations on a linear system. 

\end{abstract}


\section{Introduction}
 Due to their theoretical tractability, many reinforcement learning (RL) works have sought to use Linear Quadratic Regulator (LQR) problems \cite{anderson2007optimal} as a test-bed to better understand performance and sample complexity for both model-based (\cite{DBLP:conf/aistats/Abbasi-YadkoriL19, oymak2019non, dean2019sample}) and model-free algorithms (\cite{DBLP:conf/icml/Fazel0KM18, malik2019derivative, pmlr-v120-li20c}). In this work, we will focus on studying zeroth-order model-free algorithms for an important variant of the LQR problem. To ground our discussion, consider the following instance of LQR. Let $x_t \in \bbR^n$ and $u_t \in \bbR^k$ denote an agent's state and action respectively at time $t$, with the state dynamics evolving as $x_{t+1} = Ax_t + Bu_t.$ The agent's goal is to drive the state to a target location $x^{*}$, which we model via the following discounted infinite-horizon cost:
\begin{align}
    \begin{array}{ll}
     \displaystyle \min_{\substack{u_t, \\ t \in \bbN} }\!\!\!\! & J\!:=\!\displaystyle \bbE\left[\sum_{t=0}^{\infty} \gamma^t \left((x_t\! -\! x^{*})^\top Q(x_t \!-\! x^{*})\!+\! (u_t)^\top R u_t\right)\right]  \\
      \mbox{s.t.} & x_{t+1} = Ax_t + Bu_t, \quad x_0 \sim \mathcal{D},
    \end{array}
    \label{eq:LQR-tracking}
\end{align}
where $0 < \gamma < 1$ is a discount factor. For concreteness we assume that $\mathcal{D} = N(0, \Sigma)$, where $\Sigma$ is a positive-definite matrix. To ease technical analysis, we assume only the initial state $x_0$ is stochastic.

Most existing literature has focused on the special case when the target $x^*$ is zero, where it is known that the optimal policy is time-invariant and linear in the state $x$, of the form $u_t = Kx_t$. In pioneering work in \cite{DBLP:conf/icml/Fazel0KM18}, the authors clarified the optimization landscape of the LQR problem, and demonstrated that the vanilla LQR ($x^{*} = 0$) enjoys the \emph{gradient dominance} property, known also as the Polyak-Łojasiewicz inequality (\cite{polyak1963gradient, Lojasiewicz1963gradient}), as well as local smoothness, two crucial conditions to achieve global convergence for policy gradient. These results paved the way for the sample complexity analysis of zeroth-order policy gradient methods in both \cite{DBLP:conf/icml/Fazel0KM18} and \cite{malik2019derivative}. 

To broaden the application, we may consider allowing $x^*$ to be non-zero, as this provides greater modelling flexibility. For instance, when using LQR in robotic navigation problems or building thermal control problems, the target point is often nonzero. To distinguish this from the zero target LQR problem, we call this the \emph{LQR tracking} problem. In this setting, the optimal policy is also time-invariant, which takes the affine form $u_t = Kx_t + \tilde{K}x^{*}$, where we note that there is an additional matrix $\tilde{K}$ in the policy structure. However, to the best of our knowledge, the optimization landscape of the LQR tracking problem is still unknown. In particular, due to the change in the policy structure of the optimal policy, it is unclear if gradient dominance and local smoothness still hold for the LQR tracking problem. In light of this, analysis of model-free policy gradient methods for the LQR tracking problem has also not been developed. 

\vspace{-5pt}
\subsection{Our contributions}

First, we study the LQR tracking problem~\eqref{eq:LQR-tracking} and establish its gradient dominance property. Second, we also show that it satisfies local smoothness, which we consider in this work to be equivalent to local Lipschitzness of both the LQR tracking cost function and its gradient. 
To the best of our knowledge, this paper is the first to establish the gradient dominance and local smoothness properties for the more general LQR tracking problem, when an additional $\tilde{K}x^{*}$ is present in an agent's optimal policy and the policy is in the form of $u_t=Kx_t+\tilde{K}x^*$. These properties play a key role in showing that it is possible for policy gradient on the LQR tracking problem to achieve a globally optimal solution. Third, we propose and study the sample complexity of a model-free zeroth-order policy gradient algorithm for the LQR tracking problem, showing that it takes $\tilde{O}(\frac{d}{\ep})$ steps to reach an optimality gap of $\ep$. Here $d=2nk$ is the dimension of the policy $[K, \tilde{K}]$. The result is consistent with the sample complexity result of two-point estimators for the vanilla LQR problem in \cite{malik2019derivative} whose dimension is only $d=nk$ since the policy is only parameterized by $K$. Finally, we provide simulation results to demonstrate the convergence of our algorithm. The appendices contain all relevant additional materials not included in the conference version of this work. 
\vspace{-5pt}
\subsection{Related work}

First, as our setting is model-free, gradient estimators in our algorithms are based on zeroth-order, i.e. function value information, situating our work in the zeroth-order optimization literature  (\cite{pmlr-v120-li20c, nesterov2017random,shamir2017optimal, hajinezhad2019zone, sahu2018distributed, tang2019distributed}). Second, as discussed earlier, our work reposes on the LQR RL literature. The closest ones are (\cite{DBLP:conf/icml/Fazel0KM18, malik2019derivative}) that studies policy gradient for the vanilla LQR problem with zero target.  We also note that another line of LQR work focuses on online learning of LQR subject to time-varying disturbances or cost functions (\cite{agarwal2019logarithmic, DBLP:conf/colt/SimchowitzSH20,li2019online, li2020online}). Our work focuses on time-invariant problems but with unknown system dynamics. 

\vspace{-5pt}
\subsection{Notations}

Unless specified, $\norm{\cdot}$ refers to the Euclidean norm for vectors and Frobenius norm for matrices. In addition, for any algorithm, $\mathcal{F}_t$ denotes its filtration up to time $t$ and  $\bbE^t := \bbE[(\cdot) \mid \mathcal{F}_t]$ denotes the conditional expectation on $\mathcal{F}_t$. When appropriate, for a vector $v$, we use $(v)_K$ to refer to its subvector corresponding to the $K$ matrix and $(v)_{\tilde{K}}$ to refer to its subvector corresponding to the $\tilde{K}$ matrix. For a positive integer $m$, $[m]$ refers to the set $\{1,2,\dots,m\}$. Depending on the context, 0 could denote either the scalar 0 or the zero vector of the appropriate dimension. We also let $\rho(\cdot)$ denote the spectral radius of a matrix.

\section{Problem Setup}


For the LQR tracking problem (\ref{eq:LQR-tracking}), it can be shown that the optimal policy for the infinite-horizon LQR tracking problem takes the time-invariant form $u_t = K x_t + \tilde{K}x^*$ for some $K, \tilde{K}$ (see e.g. \cite{anderson2007optimal}). For completeness and for later use in proving properties of this problem, we provide a formal statement describing the optimal controller and optimal value (Proposition~\ref{proposition:optimal controller for J}) in the appendix. This allows us to rewrite the LQR tracking problem as follows:
\begin{align}
    \begin{array}{ll}
     \displaystyle \min_{\substack{K, \tilde{K}} }\!\!\!\! & J\!:=\!\displaystyle \bbE\left[\sum_{t=0}^{\infty} \gamma^t \left((x_t\! -\! x^{*})^\top Q(x_t \!-\! x^{*})\!+\! (u_t)^\top R u_t\right)\right]  \\
      \mbox{s.t.} & x_{t+1} = Ax_t + Bu_t, \\
      & u_t = Kx_t + \tilde{K}x^*, \quad x_0 \sim \mathcal{D}, \quad \mathcal{D} = N(0,\Sigma).
    \end{array}
    \label{eq:LQR-tracking-K-Ktilde}
\end{align}

During our learning process, we assume that matrices $A$ and $B$ are unknown but $x^*$ is available and both $x_t$ and the stage cost is observable at each time $t$.  This information setting is common in many RL applications \cite{recht2019tour}. Given a policy $(K,\tilde{K})$, and a randomly drawn initial state $x_0$, this allows us to measure the cost $J(K, \tilde{K}; x_0)$, which is a noisy realization of the (expected) local cost $J(K, \tilde{K})$. Throughout our analysis, we will also impose the following assumptions.

\vspace{-5pt}
\begin{assumption}
\label{assumption:controllable}
The dynamical system governed by the matrices $(A,B)$ is controllable \cite{anderson2007optimal}. 
\end{assumption}
\vspace{-6pt}
Thus, there exists $K \in \bbR^{k \times n}$ so that $\sqrt{\gamma}\rho(A+BK) < 1$, where $\rho(\cdot)$ measures the spectral radius of a matrix. 

\begin{assumption}
\label{assumption:access to stable K0}
We have access to a stable controller $K_0$ such that $\sqrt{\gamma}\rho(A + BK_0) < 1$.
\end{assumption}
\vspace{-6pt}
 This is an assumption commonly made in the LQR literature (\cite{dean2019sample,DBLP:conf/icml/Fazel0KM18,malik2019derivative}).

\begin{assumption}
\label{assumption: Q,R positive definite, full rank Sigma}
Cost matrices $Q$, $R$ and the covariance $\Sigma$ for the initial state are all positive definite.
\end{assumption}
\section{Optimization landscape of LQR tracking problem}
\vspace{-3pt}
We first state a result showing that a global minimizer exists (see e.g. \cite{anderson2007optimal} for similar results; we provide a tailored proof for our setup in the proof of Proposition \ref{proposition:optimal controller for J} in the appendix.)
\vspace{-5pt}
\begin{lemma}[Existence of global minimizer]
\label{lemma: existence of global minimizers}
There exists $K^{\mathrm{opt}},\tilde{K}^{\mathrm{opt}} \geq 0$ such that
\[\min_{K,\tilde{K}} J(K,\tilde{K}) = J(K^{\mathrm{opt}},\tilde{K}^{\mathrm{opt}}) := J^{\mathrm{opt}}. \]
\end{lemma}
The optimal $K^{\mathrm{opt}}, \tilde{K}^{\mathrm{opt}}$ and the minimum value $J^{\mathrm{opt}}$ are provided in Proposition~\ref{proposition:optimal controller for J} in the appendix.
\begin{lemma}[Non-convexity]
\label{lemma: non-convexity of J}
If $n \geq 2$, the LQR tracking problem~\eqref{eq:LQR-tracking-K-Ktilde} is in general non-convex.
\end{lemma} 
\vspace{-6pt}
This is a consequence of the fact that when $n \geq 2$, there exist matrices $K$ and $K'$ such that $J(K,0)$ and $J(K',0)$ are both finite but $J((K + K')/2,0)$ is not finite. An example is provided in Appendix \ref{appendix:non-convexity}. Without additional assumptions, gradient descent on a non-convex problem can only reach a stationary point, not necessarily a global minimum. However, it is known that under gradient dominance and appropriate smoothness properties, gradient descent methods, even for a non-convex problem, can find a global minimizer at a rate comparable to that of strongly convex functions \cite{karimi2016linear}. In \cite{DBLP:conf/icml/Fazel0KM18}, the authors established the gradient dominance and local smoothness of the vanilla LQR problem (where the target $x^* = 0$), when the optimization variable only involves $K$. We establish that both properties hold in the more general tracking case when the policy involves $K$ as well as $\tilde{K}$.
\begin{proposition}[Gradient dominance of LQR tracking problem]
\label{proposition:gradient dominance of lqr tracking}
There exists $\mu > 0$, depending only on $(A,B,Q,R,\Sigma,\gamma,x^*)$, such that 
\begin{align*}
   J(K,\tilde{K}) - J^{\mathrm{opt}} \leq \frac{1}{\mu}\norm*{\nabla J(K,\tilde{K})}^2.
\end{align*}
\end{proposition}
While building on the techniques in \cite{DBLP:conf/icml/Fazel0KM18}, the addition of a constant term $\tilde{K}x^{*}$ increases the complexity of the problem, and requires new analysis. We provide a proof sketch below, deferring omitted parts of the proof to appendix \ref{appendix:gradient-dominance}. 

\begin{proof}[Proof sketch of Proposition \ref{proposition:gradient dominance of lqr tracking}]
Denoting
$$c(x_t,u_t) :=  (x_t - x^*)^\top Q (x_t - x^*) + u_t^\top R u_t,$$
define the following cost-to-go function $V_{K,\tilde{K}}: \bbR^n \to \bbR_{\geq 0}$:
{\small
\begin{align*}
    & V_{K,\tilde{K}} (z) =  \sum_{t = 0}^{\infty} \gamma^t c(x_t, Kx_t+ \tilde{K}x^*), \quad \mbox{where } x_0 = z \in \bbR^n.
\end{align*}}
Supposing that $\sqrt{\gamma} \rho((A + BK))< 1$, it follows from Lemma \ref{lemma:value function of J(K,g)} in the appendix that 
\begin{align*}
    V_{K,\tilde{K}}(z) = z^\top P_K z + 2z^\top (q_{K,\tilde{K}}x^*) + r_{K,\tilde{K}}, 
\end{align*}
for some $P_K, q_{K,\tilde{K}}$ and $r_{K,\tilde{K}}$. Defining the quantities
\footnotesize
$$C_{K}\! = \!RK \!+\gamma B^\top P_K(A+BK),  d_{K,\tilde{K}} \!= \!R\tilde{K}\! +\gamma B^\top P_KB\tilde{K}\! +\gamma B^\top q_{K,\tilde{K}}\!,$$
\normalsize
we show in Lemma \ref{lemma:gradient of J} in the appendix that the gradient of $J$ can be expressed in the form
{\small
\begin{align}
    &\nabla_K J  = 2C_K \Sigma_{K,\tilde{K}} + 2d_{K,\tilde{K}}x^* \rho_{K,\tilde{K}}^\top, \label{eq:gradient_K}\\
    &\nabla_{\tilde{K} } J = \left(2C_K \rho_{K,\tilde{K}} + 2\beta_{\gamma}d_{K,\tilde{K}}x^* \right)(x^*)^\top, \label{eq:gradient_Ktilde}
\end{align}
}
where
\small
\begin{align*}
    \Sigma_{K,\tilde{K}} := \left[\sum_{t=0}^{\infty} \gamma^t x_t x_t^\top\right], \ \rho_{K,\tilde{K}} := \mathbb{E}\left[\sum_{t=0}^{\infty} \gamma^t x_t^\top \right], \ \beta_{\gamma} := \sum_{t=0}^{\infty} \gamma^t,
\end{align*}
\normalsize
and $\{x_t\}$ is the trajectory corresponding to the policy $(K,\tilde{K})$. Via a cost-difference lemma (see e.g. \cite{kakade2002approximately}), we show in Lemma \ref{lemma:cost-difference} in the appendix that
\small
\begin{align}
    J(K,\tilde{K}) - J^{\mathrm{opt}} \leq \frac{\Psi^{\mathrm{opt}}}{\sigma_{\min}(R)}\left(\norm{C_K}^2 +
    \norm*{d_{K,\tilde{K}} x^*}^2 \right)
    \label{eq:gradient_dom_opt_gap_bound}
\end{align}
\normalsize
for some $\Psi^{\mathrm{opt}}$ depending only on $(A,B,Q,R,\Sigma,\gamma,x^*)$. We now seek to bound the RHS of \eqref{eq:gradient_dom_opt_gap_bound} in terms of $\norm*{\nabla J(K,\tilde{K})}^2$. To do so, using \eqref{eq:gradient_K} and \eqref{eq:gradient_Ktilde}, we can re-express the gradient as follows:
{\small
\begin{align*}
    \nabla J(K, \tilde{K}) &= 2 \begin{bmatrix}
C_K & d_{K,\tilde{K}}
\end{bmatrix}  M_{K,\tilde{K}},
\end{align*}
}
where 
{\small
\begin{align*}
M_{K,\tilde{K}} := 
\begin{bmatrix}
\Sigma_{K,\tilde{K}} &  \rho_{K,\tilde{K}}(x^*)^\top \\
 x^*(\rho_{K,\tilde{K}})^\top & \beta_{\gamma}  x^* (x^*)^\top
\end{bmatrix}.
\end{align*}
}
Letting $A_K := A+BK$ to ease the notation, note that
{\small
\begin{align*}
    M_{K,\tilde{K}} &= \begin{bmatrix}
    \Sigma_{K,\tilde{K}} & \rho_{K,\tilde{K}}(x^*)^\top \\
    x^*(\rho_{K,\tilde{K}})^\top & \beta_{\gamma} x^* (x^*)^\top
    \end{bmatrix} \\
    &= \sum_{t=0}^{\infty} \gamma^t 
    \begin{bmatrix}
    \mathbb{E}[x_t (x_t)^\top] & \mathbb{E}[x_t](x^*)^\top \\
    x^* \mathbb{E}[x_t]^\top & x^* (x^*)^\top
    \end{bmatrix} \\
    &= \sum_{t=0}^{\infty} \gamma^t 
    \begin{bmatrix}
    \mathbb{E}[x_t]\mathbb{E}[x_t]^\top + \mathbb{E}[A_K^t x_0 x_0^\top (A_K^t)^\top] & \mathbb{E}[x_t](x^*)^\top \\
    x^* \mathbb{E}[x_t]^\top & x^* (x^*)^\top
    \end{bmatrix} \\
    &\labelrel\succeq{eq:M_K_Ktilde_penultimate-main}
    \begin{bmatrix}
    \Sigma & 0_{n \times n} \\
    0_{n \times n} & x^* (x^*)^\top
    \end{bmatrix} \labelrel\succeq{eq:M_K_Ktilde_last-main}
    \begin{bmatrix}
    \alpha I_{n \times n} & 0_{n \times n} \\
    0_{n \times n} & x^* (x^*)^\top
    \end{bmatrix}.
\end{align*}
}
To obtain \eqref{eq:M_K_Ktilde_penultimate-main}, we used the fact that $\bbE\left[x_0 x_0^\top \right] = \Sigma$, $\bbE[x_0] = 0$, as well as the fact that for each $t$,
{\small
\begin{align*}
    &\begin{bmatrix}
    \mathbb{E}[x_t]\mathbb{E}[x_t]^\top + \mathbb{E}[A_K^t x_0 x_0^\top (A_K^t)^\top] & \mathbb{E}[x_t](x^*)^\top \\
    x^* \mathbb{E}[x_t]^\top & x^* (x^*)^\top
    \end{bmatrix} \\
    &= \! \begin{bmatrix}
    \mathbb{E}[x_t] \\
    x^*
    \end{bmatrix} \begin{bmatrix}
    \mathbb{E}[x_t]^\top & (x^*)^\top 
    \end{bmatrix} \! + \! \begin{bmatrix}
    \mathbb{E}[A_K^t x_0 x_0^\top (A_K^t)^\top] & 0_{n \times n} \\
    0_{n \times n} & 0_{n \times n}
    \end{bmatrix} \succeq 0.
\end{align*}
}
To obtain \eqref{eq:M_K_Ktilde_last-main}, we used the positive definite assumption on $\Sigma$, ensuring that $\Sigma \succeq \alpha I$  for some $\alpha > 0$.
{\small
Hence,
\begin{align*}
\norm*{\nabla J(K, \tilde{K})}^2 &= 4 \norm*{\begin{bmatrix}
C_K & d_{K,\tilde{K}}
\end{bmatrix}  M_{K,\tilde{K}}}^2 \\
&\geq 4 \norm*{\begin{bmatrix}
C_K & d_{K,\tilde{K}}
\end{bmatrix}  \begin{bmatrix}
    \alpha I_{n \times n} & 0_{n \times n}\\
    0_{n \times n} & x^* (x^*)^\top
    \end{bmatrix}}^2 \\
&= 4 \left(\alpha^2 \norm*{C_K}^2 + \norm{x^*}^2 \norm*{d_{K,\tilde{K}}x^*}^2\right).
\end{align*}
}
Suppose first that $x^* \neq 0$. Then,
\small
\begin{align*}
    &J(K, \tilde{K}) - J(K^{\mathrm{opt}}, \tilde{K}^{\mathrm{opt}}) \\
    &\labelrel\leq{eq:gradient_dominant_cite_opt_gap-main} \frac{\Psi^{\mathrm{opt}}}{\sigma_{\min}(R)} \left(\norm*{C_K}^2 + \norm*{d_{K,\tilde{K}}x^*}^2 \right) \\
    &\leq \frac{\Psi^{\mathrm{opt}}}{4\min\{\alpha^2, \norm*{x^*}^2\} \sigma_{\min}(R)} \norm*{\nabla J(K, \tilde{K})}^2,
\end{align*}
\normalsize
where we note that \eqref{eq:gradient_dominant_cite_opt_gap-main} follows from   \eqref{eq:gradient_dom_opt_gap_bound}.
In the case when $x^* = 0$,
\begin{align*}
    J(K, \tilde{K}) - J(K^{\mathrm{opt}}, \tilde{K}^{\mathrm{opt}}) &\leq \frac{\Psi^{\mathrm{opt}}}{\sigma_{\min}(R)} \left(\norm*{C_K}^2 \right) \\
    &\leq \frac{\Psi^{\mathrm{opt}}}{4\alpha^2 \sigma_{\min}(R)} \norm*{\nabla J(K, \tilde{K})}^2.
\end{align*}
This completes the proof sketch.
\end{proof}

Gradient dominance alone cannot ensure convergence to a global minimum. As noted in (\cite{DBLP:conf/icml/Fazel0KM18, malik2019derivative}), a function needs to also exhibit local smoothness properties. We note that in showing the Lipschitz result for the sample cost below, we make the simplifying assumption that the initial distribution of $x_0$ is bounded; an extension to sub-Gaussian random distribution is possible by appealing to high-probability bounds and standard truncation arguments, as suggested in \cite{malik2019derivative}. Thus, in the sequel, we will work with the assumption $\norm{x_0} \leq C_n$ for all $x_0 \sim \mathcal{D}$ for some $C_n > 0$.

\begin{proposition}[Local smoothness of LQR tracking problem]
\label{proposition: local smoothness}
Let $\mathcal{G}_C = \{(K,\tilde{K}) \mid J(K,\tilde{K}) \leq C \}$ be a sublevel set of $J$, where $C > 0$. Suppose $\norm{x_0} \leq C_n$ for all $x_0 \sim \mathcal{D}$ for some $C_n > 0$. Then, there exists a local radius $\rho > 0$, Lipschitz parameter $\lambda > 0$ and Lipschitz gradient parameter $L > 0$, which all depend only on $(A,B,Q,R,\Sigma, \gamma,x^*,C,C_n)$, such that for any $(K,\tilde{K})$ in $\mathcal{G}_C$, whenever $\norm{(K',\tilde{K}') - (K,\tilde{K})} \leq \rho$, 
{\small
\begin{align*}
     &\abs*{J(K',\tilde{K}') - J(K,\tilde{K})} \leq \lambda \norm*{(K',\tilde{K}') - (K,\tilde{K})}, \\
     &\abs*{J(K',\tilde{K}';x_0) - J(K,\tilde{K};x_0)} \leq \lambda \norm*{(K',\tilde{K}') - (K,\tilde{K})} \quad \forall x_0 \sim \mathcal{D}, \\
    &\norm*{\nabla\! J(K',\tilde{K}')\! -\! \nabla J(K,\tilde{K})}\! \leq \!L\norm*{(K',\tilde{K}')\!- \!(K,\tilde{K})}. 
\end{align*}
}
\end{proposition}

\begin{proof}[Proof sketch for Proposition \ref{proposition: local smoothness}]
 Suppose $(K,\tilde{K})$ is in a sublevel set $\mathcal{G}_C$ of $J$ for some $C > 0$, and let $(K',\tilde{K}')$ denote another policy.
For notational convenience, define
\footnotesize
\begin{align*}
   &\Psi_{(\!K,\tilde{K}\!),(\!K',\tilde{K}'\!)} := \norm*{R \!+\! \gamma B^\top P_K B}\!\left(\max\!\left\{\!\norm*{\Sigma_{K',\tilde{K}'}}\!,\! \norm*{\rho_{K',\tilde{K}'}}\!,\! \norm*{x^*}^2 \!\right\}\right).
\end{align*}
\normalsize
Then, applying the cost-difference lemma \cite{kakade2002approximately}, we can show
\small
\begin{align*}
    &J(K',\tilde{K}') - J(K,\tilde{K}) \\
    &= 2 \tr((K'-K)^\top (C_K \underline{\bm{\Sigma_{K',\tilde{K}'}}}  + d_{K,\tilde{K}}x^* \underline{\bm{\rho_{K',\tilde{K}'}}}^\top)) \\
    & \ \ + 2\tr((\tilde{K}' - \tilde{K})^\top (C_K \underline{\bm{\rho_{K',\tilde{K}'}}} + d_{K,\tilde{K}}x^*)(x^*)^\top) \\
    & \ \ + O\left( \Psi_{(K,\tilde{K}),(K',\tilde{K}')} \left(\norm*{K' - K}^2 + \norm*{\tilde{K}' - \tilde{K}}^2 \right)\right) \\
    &\labelrel\approx{eq:smoothness_approx-main} 2 \tr((K'-K)^\top (C_K \underline{\bm{\Sigma_{K,\tilde{K}}}}  + d_{K,\tilde{K}}x^* \underline{\bm{\rho_{K,\tilde{K}}}}^\top)) \\
    & \ \ + 2\tr((\tilde{K}' - \tilde{K})^\top (C_K \underline{\bm{\rho_{K,\tilde{K}}}} + d_{K,\tilde{K}}x^*)(x^*)^\top) \\
    & \ \ + O\left( \Psi_{(K,\tilde{K}),(K',\tilde{K}')} \left(\norm*{K' - K}^2 + \norm*{\tilde{K}' - \tilde{K}}^2 \right)\right) \\
    &\labelrel={eq:gradient_equality-main} \tr((K'-K)^\top \nabla_K J(K,\tilde{K}) \\
    & \ \ + \tr((\tilde{K}' - \tilde{K})^\top \nabla_{\tilde{K}} J(K,\tilde{K}) \\
    & \ \ + O\left( \Psi_{(K,\tilde{K}),(K',\tilde{K}')} \left(\norm*{K' - K}^2 + \norm*{\tilde{K}' - \tilde{K}}^2 \right)\right) \\
    &\labelrel={eq:smoothness_simplify_big_O-main}\tr((K'-K)^\top \nabla_K J(K,\tilde{K}) \\
    & \ \ + \tr((\tilde{K}' - \tilde{K})^\top \nabla_{\tilde{K}} J(K,\tilde{K}) \\
    & \ \ + O\left( \Psi_C \left(\norm*{K' - K}^2 + \norm*{\tilde{K}' - \tilde{K}}^2 \right)\right)
\end{align*}
\normalsize
Above, we bolded and underlined the terms that are different before and after the approximation in \eqref{eq:smoothness_approx-main}. By constraining $\norm*{K' - K}$ and $\norm*{\tilde{K}' - \tilde{K}}$ to both be sufficiently small, it is possible to prove perturbation bounds of the form
\small
\begin{align}
    &\norm*{\Sigma_{K',\tilde{K}'} -  \Sigma_{K,\tilde{K}}} \leq c \left(\norm*{K' - K} + \norm*{\tilde{K}' - \tilde{K}} \right), \label{eq:perturbation_bd_1}\\ 
    &\norm*{\rho_{K',\tilde{K}'} - \rho_{K,\tilde{K}}} \leq c' \left(\norm*{K' - K} + \norm*{\tilde{K}' - \tilde{K}} \right) \label{eq:perturbation_bd_2}
\end{align}
\normalsize
for some positive constants $c,c'$ depending only on $(A,B,Q,R,\Sigma, \gamma,x^*,C)$. This shows that the approximation in \eqref{eq:smoothness_approx-main} is locally valid. Meanwhile, the equality in \eqref{eq:gradient_equality-main} holds by the form of the gradient of $J$ we compute in Lemma \ref{lemma:gradient of J}. Finally, for the equality in \eqref{eq:smoothness_simplify_big_O-main}, we replaced the constant $\Psi_{(K,\tilde{K}),(K',\tilde{K}')}$ with a constant $\Psi_C$, which can be shown to depend only on the sublevel set parameter $C$ as well as the system parameters $(A,B,Q,R,\Sigma,\gamma,x^*)$. To see why, recall that
\footnotesize
\begin{align*}
   &\Psi_{(\!K,\tilde{K}\!),(\!K',\tilde{K}'\!)} := \norm*{R \!+\! \gamma B^\top P_K B}\!\left(\max\!\left\{\!\norm*{\Sigma_{K',\tilde{K}'}}\!,\! \norm*{\rho_{K',\tilde{K}'}}\!,\! \norm*{x^*}^2 \!\right\}\right).
\end{align*}
\normalsize
Since $(K,\tilde{K})$ lies within the sublevel set $\mathcal{G}_C$, it is possible to uniformly bound the norm of $P_K$, $\Sigma_{K,\tilde{K}}$, as well as $\rho_{K,\tilde{K}}$, over the set $\mathcal{G}_C$. Coupled with perturbation bounds such as the ones in \eqref{eq:perturbation_bd_1} and \eqref{eq:perturbation_bd_2}, it is not hard to see that $\Psi_{(K,\tilde{K}),(K',\tilde{K}')}$ can be replaced with a constant $\Psi_C$, which depends only on the sublevel set parameter $C$ as well as the system parameters $(A,B,Q,R,\Sigma,\gamma,x^*)$. From \eqref{eq:smoothness_simplify_big_O-main}, we see that for any $(K,\tilde{K})$ in $\mathcal{G}_C$, the cost difference between it and a sufficiently close policy $(K',\tilde{K}')$ can be written as an inner product with the gradient plus some quadratic terms. By deriving norm bounds on $\nabla J(K,\tilde{K})$ uniform over any $(K,\tilde{K})$ over $\mathcal{G}_C$, it follows that $J$ is locally Lipschitz within a radius $\rho$ and Lipschitz parameter $\lambda$ that hold uniformly over $\mathcal{G}_C$. We can prove the local Lipschitzness of the sample cost and the gradient of $J$ with a similar recipe: prove a variety of norm bounds over terms such as $P_K$, $\Sigma_{K,\tilde{K}}, \rho_{K,\tilde{K}}$ uniformly over $\mathcal{G}_C$, and show perturbation norm bounds on difference terms between quantities that depend on $(K',\tilde{K}')$ and quantities that depend on $(K,\tilde{K})$ such as the following bound:
\begin{align*}
    &\norm*{\Sigma_{K',\tilde{K}'} -  \Sigma_{K,\tilde{K}}} \leq c \left(\norm*{K' - K} + \norm*{\tilde{K}' - \tilde{K}} \right).
\end{align*}
The full proof can be found in appendix \ref{appendix:local-smoothness}.
\end{proof}


\section{Algorithm design}

\subsection{Background on zeroth-order model-free LQR learning}
As discussed earlier, the optimal policy to minimize the  infinite-horizon cost \eqref{eq:LQR-tracking} is given by $u_t = Kx_t + \tilde{K}x^*$, i.e. an affine time-invariant policy. Denoting $\hat{K} := (K,\tilde{K})$ to ease the notation, policy gradient on $\hat{K}$ then involves iterating the update $\hat{K}_{t+1} = \hat{K}_t - \eta z_t$, where $z_t$ is an estimate of the policy gradient $\nabla J (\hat{K}_t)$, and $\eta > 0$ is a step-size. Exactly computing the gradient $\nabla J(\hat{K}_t)$ requires knowledge of the system matrices $(A,B)$. When $(A,B)$ is unknown,  (\cite{DBLP:conf/icml/Fazel0KM18, malik2019derivative}) have proposed a zeroth-order mechanism to estimate the gradient for the vanilla LQR problem with just $K$. In this line of work, the estimate $z_t$ is calculated based on function value information about the cost $J(\cdot)$ near the current policy $K_t$. We adapt existing zeroth-order estimators proposed in the literature to our setup, where the policy now includes not just the matrix {\small$K$} but also {\small$\tilde{K}$}. 
In particular, we consider the following symmetric two-point estimator:
\small
\begin{align*}
  z_r(\hat{K};x_0, \delta)\!:=\! d\frac{J(\hat{K} + r\delta; x_0)\! -\! J(\hat{K} - r\delta;x_0)}{2r} \delta,  \delta \sim \mathrm{Unif}(\bbS^{d - 1}).
\end{align*}
\normalsize
Above, $r > 0$ can be viewed as a smoothing radius, and $\delta$ is a random perturbation vector drawn at random from the unit sphere $\mathrm{Unif}(\bbS^{d-1})$, where we denote $d := 2nk$, which is the dimension of $\hat{K} = (K,\tilde{K})$. The multiplicative factor $d$ is necessary so that the estimator has the same scaling as the true gradient $\nabla J(\hat{K})$. The symmetric two-point estimator resembles a finite-difference approximation of the gradient, and as such is clearly motivated. We point out here a technical caveat --- the notation $J(\hat{K} \pm r\delta; x_0)$ means that the two measurements are based on the policies $\hat{K} \pm r\delta$ for a common random initial state $x_0$. In cases when we cannot obtain two neighboring function evaluations corresponding to the same noisy initial state $x_0$, a one-point estimator based on a single function evaluation may be required (\cite{pmlr-v120-li20c, flaxman2004online})\footnote{In this work, we focus on studying the two-point estimator. Analysis of the one-point estimator setting will be left to future work.}.

\vspace{-5pt}
\subsection{Proposed algorithm: zeroth-order LQR tracking}
Our goal is to approximately optimize $J$ using zeroth-order information to form policy gradient updates. We consider incorporating a mini-batching procedure to reduce the variance of the gradient estimators. We state the algorithm below, and highlight its key features.

\SetInd{0.1em}{0.01em}
\begin{algorithm}
\caption{Zeroth-order LQR tracking algorithm}
\label{alg:zeroth-order lqr}
\DontPrintSemicolon
\SetAlgoNoLine
Given: iteration number $T \geq 1$, mini-batch size $m \geq 1$, initial stable $K_0$, initial $\tilde{K}_0$, step size $\eta > 0$, and smoothing radius $r  > 0$. Denote $d := 2nk$. \;
\For{$\mbox{iteration }t = 0,1\!,\! \dots\!,\! T-1$}{
Sample \small$\{(x_0)_t^i\}_{i=1}^m \!\sim\!N(0,\Sigma), \ \  \{\delta_t^i\}_{i=1}^m\!\sim\!\mathrm{Unif}(\bbS^{d - 1}).$\normalsize \;
Compute zeroth-order gradient estimator $z_t^i$ for each $i \in [m]$:
\begin{align*}
z_t^i\!\leftarrow\! z_r((K_t,\tilde{K}_t); (x_0)_t^i, \delta_t^i)
\end{align*}\;
\vspace{-8pt}
Set $z_t \leftarrow \frac{1}{m} \sum_{i=1}^m z_t^i$. Update $K_t, \tilde{K}_t$: 
\begin{align*}
&K_{t+1} \leftarrow K_t - \eta (z_t)_K, \\ 
&\tilde{K}_{t+1} \leftarrow \tilde{K}_t - \eta (z_t)_{\tilde{K}},
\end{align*}
where $(z_t)_K, (z_t)_{\tilde{K}}$ mean the gradient estimator value for $K$ and $\tilde{K}$ respectively.
}
\Return{$K_T, \ \tilde{K}_T$}
\end{algorithm} 
\begin{itemize}[leftmargin=12pt]
\itemsep=3pt
\item \textit{Initial conditions (Line 1).} We assume knowledge of a stable $K_0 \in \bbR^{k \times n}$. The choice of initial constant 
term $\tilde{K}_0 \in \bbR^{k \times n}$ is flexible, and can be picked to be the zero vector without prior information. We consider a constant step-size $\eta$ and smoothing radius $r$. 
    
\item\textit{Zeroth-order gradient estimator (Lines 3 and 4).} We perform sampling for $(x_0)_t^i$ and $\delta_t^i$ on Line 3, where we note that the sampling procedure is independent for each sample $i$. We adopt a two-point zeroth-order gradient estimator, as seen in Line 4 of the algorithm. 
\vspace{-2pt}
\item\textit{Minibatch averaging and update (Line 5)} At the end of each iteration, we average over the $m$ independent zeroth-order gradient estimators in the mini-batch to form an estimator $z_t \leftarrow \frac{1}{m} \sum_{i=1}^m z_t^i$, and update the policy $(K_t, \tilde{K}_t)$ accordingly. 

\end{itemize}

\section{Main results}
In this section, we will present the global convergence of the proposed policy gradient algorithm. Our results can be considered as a special case of the general convergence results for stochastic zeroth-order algorithms applied to locally smooth and gradient dominant non-convex problems derived in \cite{malik2019derivative}, but for completeness we present the proof here.
\subsection{Convergence of Algorithm~\ref{alg:zeroth-order lqr}}
\textbf{Preliminaries.} We first introduce several quantities which will appear in the main results. First, we define the stability region $\mathcal{G}_{10J_0}$ which we show later the iterates of Algorithm \ref{alg:zeroth-order lqr} stay within with large probability.  Let
\small
\begin{align*}
    &\mathcal{G}_{10J_0} := \left\{(K,\tilde{K}): J(K,\tilde{K}) \leq 10J_0) \right\}, \quad J_0 := J(K_0,\tilde{K}_0).
\end{align*}
\normalsize
Using the notation $\hat{K} = (K,\tilde{K})$, we define
\small
\begin{align*}
    &Z_{\infty} := \sup_{\substack{\hat{K} \in \mathcal{G}_{10J_0} \\ x_0,\delta}} \left\{\norm*{z_r(\hat{K}; x_0,\delta)}\right\}, \\
    &Z_{2}\! := \!\sup_{\hat{K} \in \mathcal{G}_{10J_0}}\left\{\!\bbE\! \left[\norm*{z_r(\hat{K};x_0,\delta) - \bbE \left[z_r(\hat{K};x_0,\delta)  \mid \hat{K} \right] }^2\right]\right\}\!,
\end{align*}
\normalsize
which form bounds on the maximum size and variance of the zeroth-order gradient update respectively.
Next, by Proposition \ref{proposition: local smoothness}, note that there exists a local radius $\rho > 0$, local Lipschitz parameter $\lambda > 0$ and Lipschitz gradient parameter $L > 0$ such that  if $(K,\tilde{K}) \in \mathcal{G}_{10J_0}$,  if $\norm{(K', \tilde{K}') - (K,\tilde{K})} \leq \rho,$ then
\small
\begin{align*}
    &  \abs*{ J(K', \tilde{K}') - J (K,\tilde{K})} \leq \lambda \norm*{(K', \tilde{K}') - (K,\tilde{K})}, \\
    &\abs*{J(K',\tilde{K}';x_0) - J(K,\tilde{K};x_0)} \leq \lambda \norm*{(K',\tilde{K}') - (K,\tilde{K})} \quad \forall x_0 \sim \mathcal{D}, \\
    & \norm*{\nabla J(K', \tilde{K}') - \nabla J (K,\tilde{K})}\leq L \norm*{(K', \tilde{K}') - (K,\tilde{K})}.
\end{align*}
\normalsize
These preliminary definitions pave the way for our main results. We begin with stating Theorem \ref{alg:zeroth-order lqr}, showing the convergence of the zeroth-order tracking algorithm.

\begin{theorem}[Convergence of LQR tracking]
\label{theorem:lqr_tracking_algorithm_convergence}
Suppose the step-size $\eta$, smoothing radius $r$ are chosen to satisfy
{\small
\begin{align*}
    \eta \leq \min \left\{\frac{m \ep \mu }{240L Z_2}, \frac{1}{4L},  \frac{\rho}{Z_{\infty}}\right\}, r \leq \min \left\{ \sqrt{\frac{\ep \mu}{240 L^2}}, \rho \right\}.
\end{align*}
}
Then, if the error tolerance $\ep$ satisfies $\ep\log(120\Delta_0/\ep)) \leq 5\Delta_0$, the iterate $(K_t,\tilde{K}_t)$ produced by Algorithm \ref{alg:zeroth-order lqr} satisfies 
\begin{align*}
    J(K_T, \tilde{K}_T) - J^{\mathrm{opt}} \leq \ep,
\end{align*}
when $T = \frac{4}{\mu \eta} \log(120\Delta_0/\ep)$ steps,
with probability at least 3/4, where $\Delta_0 = J(K_0, \tilde{K}_0) - J^{\mathrm{opt}}$.
\end{theorem}

This gives rise to the following corollary.
\begin{corollary}[Sample complexity]
\label{corollary:sample_complexity_2pt}
To reach an $\ep$-optimality gap with probability at least 3/4, subject to the assumptions in Theorem \ref{theorem:lqr_tracking_algorithm_convergence}, Algorithm \ref{alg:zeroth-order lqr} requires $$T = \tilde{O}\left(\frac{d}{\min\{m\mu \ep/\lambda^2, d/(L), \rho\mu /\lambda\}}\right) \mbox{steps},$$
where $\tilde{O}$ hides log terms and $d=2nk$.
\end{corollary}
\begin{proof}[Proof sketch of Corollary \ref{corollary:sample_complexity_2pt}]
The key to showing the corollary is to upper bound $Z_{\infty}$ and $Z_2$, defined earlier, which bound the maximum size and the variance of the zeroth-order estimators respectively. Using techniques from zeroth-order optimization (e.g. \cite{shamir2017optimal}), we can prove that the following bounds
\begin{align*}
    Z_{\infty} \leq d\lambda , \quad Z_{2} \leq d\lambda^2
\end{align*}
hold, where we recall $\lambda >0$ is the (uniform) local Lipschitz parameter over $\mathcal{G}_{10J_0}$, and $d = 2nk$ is the dimension of the optimization problem. The full proof can be found in Appendix \ref{appendix:sample-complexity}.
\end{proof}
We now discuss the implications of Theorem \ref{theorem:lqr_tracking_algorithm_convergence} and Corollary \ref{corollary:sample_complexity_2pt}.

\paragraph{Probabilistic convergence}
Inherently, maintaining stability of $K$, \emph{i.e.} $\rho(\sqrt{\gamma}(A+BK)) < 1$, is a critical issue for LQR learning, since when $K$ is unstable, the infinite-horizon LQR cost can diverge \cite{DBLP:conf/icml/Fazel0KM18}. This stability requirement is in tension with the constant accumulation of noise in the learning procedure due to the stochastic zeroth-order updates. For this reason, the convergence result holds with a constant probability, namely 0.75. However, we note that by allowing the minibatch size to scale with $\tilde{O}(1/\ep)$, it is possible to improve the convergence to a high-probability result. By following analysis in Theorem 2 of \cite{malik2019derivative}, it can be shown that by allowing the minibatch size to scale as $\tilde{O}(\frac{\log(1/\delta)}{\ep})$ for a user-chosen $\delta > 0$, the per-iteration decrease in cost which holds in expectation in Lemma \ref{lemma: per-iteration-change-in-optimality gap-first-time} can in fact hold with probability $1 - \delta$. Then, over the course of the algorithm with $T$ steps it can be shown that the algorithm is stable with probability $1 - T\delta$. Since $\delta > 0$ is user-chosen, this represents a high-probability convergence result. We leave careful analysis of this issue to future work.

\paragraph{Dependence on $d$ and $\ep$} From Corollary \ref{corollary:sample_complexity_2pt}, we note that the sample complexity displays (1) a linear dependence on $d$ , and (2) an $\ep^{-1}$ scaling. These observations are consistent with the sample complexity results for two-point estimators for the zero-target LQR in \cite{malik2019derivative} (note the dimension of our optimization problem is double that of the vanilla LQR due to having an extra $\tilde{K}$).

\paragraph{Speedup with increasing minibatch size}We note that the convergence speedup with the minibatch size $m$ may scale as $\tilde{O}(m)$. This is because the averaging process reduces the estimator variance, thus potentially allowing us to increase the step-size $\eta$ by a factor of $m$. However, due to stability issues, there exists some maximum step-size which permits convergence. In particular, $\eta \leq \tilde{O}\left(\frac{1}{L}, \frac{\rho\mu}{d\lambda}\right)$ must hold. We now proceed to discuss the proof for Theorem \ref{theorem:lqr_tracking_algorithm_convergence}.


\subsection{Proof of Theorem~\ref{theorem:lqr_tracking_algorithm_convergence}}
 At a high level, the proof comprises the following steps.
\begin{enumerate}
    \item Showing an expected decrease in cumulative cost every iteration.
    \item Showing that the sequence $(K_t,\tilde{K}_t)$ stays in $\mathcal{G}_{10J_0}$ throughout the algorithm with large probability, via a martingale argument.
    \item Prove convergence by combining the above two results.
\end{enumerate}

We first show a per-iteration change in optimality gap for the joint sequence, which will be essential in the convergence proofs. For convenience, define the optimality gap
\begin{align*}
    \Delta_t := J(K_t,\tilde{K}_t) - J^{\mathrm{opt}}.
\end{align*}
\begin{restatable}[Per-iteration change in optimality gap]{lemma}{perIterOptGap}\label{lemma: per-iteration-change-in-optimality gap-first-time}
Suppose $(K_t,\tilde{K}_t) \in \mathcal{G}_{10J_0}$. Let $m$ be the mini-batch size.
Then, if we choose step-size $\eta > 0$, smoothing radius $r > 0$ such that
\begin{align*}
    \eta \leq \min \left\{\frac{\rho}{Z_{\infty}}, \frac{1}{4L}\right\}, \quad r \leq  \min\left\{\sqrt{\frac{\ep\mu}{240L^2}}, \rho \right\},
\end{align*}
the optimality gap satisfies the following bound
\begin{align*}
    \mathbb{E}^t [\Delta_{t+1}] \leq \left(1 - \frac{\eta\mu}{4}\right) \Delta_t + \frac{\eta^2L}{2}\left(\frac{Z_2}{m} \right) + \frac{\eta \mu}{120}\ep.
\end{align*}
\end{restatable}
We defer the proof of Lemma \ref{lemma: per-iteration-change-in-optimality gap-first-time} to Appendix \ref{appendix:optimality-gap}, but note that gradient dominance (Proposition \ref{proposition:gradient dominance of lqr tracking}) and local smoothness (Proposition \ref{proposition: local smoothness}) play key roles in the proof. We next show that with a large (but constant) probability, the joint sequence $(K_t,\tilde{K}_t)$ produced by the algorithm remains in the stable region $\mathcal{G}_{10J_0}$, satisfying the assumption necessary for Lemma \ref{lemma: per-iteration-change-in-optimality gap-first-time} to hold.
\vspace{-5pt}

\begin{restatable}[Stability of algorithm]{proposition}{stability}
\label{proposition:stability-first-time}
Suppose 
\begin{align*}
    \eta\! \leq\! \min \left\{\frac{\mu \ep m}{240LZ_2}, \frac{\rho}{Z_{\infty}}, \frac{1}{4L}\right\}, \ \ r \!\leq\!  \min\left\{\sqrt{\frac{\ep\mu}{240L^2}}, \rho \right\}\!,
\end{align*}
and that $\ep > 0$ is small enough such that
\begin{align*}
    \ep \log\left(\frac{120\Delta_0}{\ep} \right) \leq 5 \Delta_0.
\end{align*}
Then, with probability larger than 4/5, $(K_t,\tilde{K}_t)$ remains in the region $\mathcal{G}_{10J_0}$ for the duration of the algorithm.
\end{restatable}
\begin{proof}
Suppose we run the algorithm for $T$ iterations. 
Define a stopping time $\tau := \min_{0 \leq t \leq T} \{\Delta_t > 10J_0\}$. Consider the sequence  $\{Y_{t}\}_{t=0}^T$:
\begin{align*}
    Y_t = \Delta_{t \wedge \tau} + (T - t) \left[ \frac{\eta^2L}{2}\frac{Z_2}{m} + \frac{\eta\mu }{120}\ep \right], \quad t = 0,1,\dots,T.
\end{align*}
We now show that this forms a non-negative supermartingale sequence. Non-negativity and adaptation of $Y_t$ to the filtration $\mathcal{F}_t$ are clear to see. We now prove integrability of $\{Y_t\}_{t=0}^T$. Note that
\small
\begin{align*}
   \bbE\left[\Delta_{t \wedge \tau}\right] &= \bbE\left[\Delta_{t \wedge \tau} 1_{t < \tau}\right] + \bbE\left[\Delta_{t \wedge \tau} 1_{t \geq \tau}\right] \\
   &\labelrel={eq:mart_int_0}  \bbE\left[\Delta_t 1_{t < \tau}\right] + \bbE \left[\sum_{i=1}^t \Delta_i 1_{\tau = i}\right] \\
   &\labelrel\leq{eq:mart_int_i} 10 J_0 +  \sum_{i=1}^t \bbE \left[\bbE^{i-1}[\Delta_{i}1_{\tau \geq i}] \right] \\
   &\labelrel\leq{eq:mart_int_ii} 10 J_0 + \sum_{i=1}^t \bbE \left[\bbE^{i-1}[\Delta_{i}]1_{\tau \geq i} \right] \\
   &\labelrel\leq{eq:mart_int_iii} 10 J_0 + \sum_{i=1}^t \bbE\left[\left(\left( 1 - \frac{\eta \mu}{4}\right)\Delta_{i-1} + \left[ \frac{\eta^2L}{2}\frac{Z_2}{m} + \frac{\eta\mu }{120}\ep \right]  \right) 1_{\tau \geq i}\right] \\
   &\labelrel\leq{eq:mart_int_iv} 10 J_0 + \sum_{i=1}^t \bbE\left[\Delta_{i-1}1_{\tau \geq i} + \left[ \frac{\eta^2L}{2}\frac{Z_2}{m} + \frac{\eta \mu}{120}\ep \right]  \right] \\
   &\labelrel\leq{eq:mart_int_v} 10tJ_0 + t\left[ \frac{\eta^2L}{2}\frac{Z_2}{m} + \frac{\eta \mu}{120}\ep \right] < \infty
\end{align*}
\normalsize
Above, \eqref{eq:mart_int_0} holds since $1_{t \geq \tau} = \sum_{i=1}^t 1_{\tau = i}$, while \eqref{eq:mart_int_i} is a consequence of the definition of $\tau$, which ensures that $\Delta_t 1_{t < \tau} \leq 10 J_0$. Meanwhile, \eqref{eq:mart_int_ii} holds since $1_{\tau \geq i}$ is adapted to $\mathcal{F}_{t-1}$. \eqref{eq:mart_int_iii} holds due to Lemma \ref{lemma: per-iteration-change-in-optimality gap-first-time}, while \eqref{eq:mart_int_iv} holds since $\eta, \mu \geq 0$. Finally, \eqref{eq:mart_int_v} holds due to our choice of $\tau$. Since $\left[ \frac{\eta^2L}{2}\frac{Z_2}{m} + \frac{\eta \mu}{120}\ep \right] < \infty$ holds, and $\Delta_t \geq 0$, it follows that each $Y_t$ is integrable.

We proceed to show that $\bbE^t[Y_{t+1}] \leq Y_t$. For notational convenience, we define 
\small
$$N_2 := \frac{\eta^2L}{2}\frac{Z_2}{m} + \frac{\eta \mu }{120}\ep  $$\normalsize
in the proofs. Observe that
{\small
\begin{align*}
   & \bbE^t[Y_{t+1}] \\
    &= \bbE^t\left[ \Delta_{(t+1) \wedge \tau} 1_{\tau \leq t}\right] 
    + \bbE^t\left[ \Delta_{(t+1) \wedge \tau} 1_{\tau > t}\right]+ (T- (t+1)) N_2 \\
    &\labelrel={eq:supermart_i-main} \bbE^t\left[ \Delta_{t \wedge \tau} 1_{\tau \leq t}\right] + \bbE^t\left[ \Delta_{t+1} 1_{\tau > t}\right] + (T- (t+1)) N_2 \\
    &\labelrel={eq:supermart_ii-main} \Delta_{t \wedge \tau} 1_{\tau \leq t} + \bbE^t\left[ \Delta_{t+1} \right]1_{\tau > t} + (T- (t+1)) N_2 \\
    &\labelrel\leq{eq:supermart_iii-main}\Delta_{t \wedge \tau} 1_{\tau \leq t} + \left( \left(1 - \frac{\eta \mu}{4}\right) \Delta_t + N_2\right) 1_{\tau > t} + (T- (t+1)) N_2 \\
    &\leq \Delta_{t \wedge \tau} + 
    (T-t)N_2 = Y_t
\end{align*}
}
Above, \eqref{eq:supermart_i-main} holds since $\Delta_{(t+1)\wedge \tau }1_{\tau \leq t} = \Delta_{t \wedge \tau} 1_{\tau \leq t}$, and similarly $\Delta_{(t+1) \wedge \tau} 1_{\tau > t} = \Delta_{t+1} 1_{\tau > t}$. Meanwhile, \eqref{eq:supermart_ii-main} holds since $\Delta_{t \wedge \tau}$ and $1_{\tau \leq t}$ are both adapted to $\mathcal{F}_t$, while \eqref{eq:supermart_iii-main} holds due to Lemma \ref{lemma: per-iteration-change-in-optimality gap-first-time}. 
Thus $\{Y_t\}_{t=0}^T$ forms a non-negative supermartingale sequence. Then via a maximal 
inequality for non-negative supermartingales (see e.g. Exercise 4.8.2 in Version 5 of \cite{durrett2019probability}), we have that
{\small
\begin{align*}
    \bbP\left(\max_{0 \leq t \leq T} Y_t \geq \nu \right) &\leq \frac{\mathbb{E}[Y_{0}]}{\nu} \\
    &\leq \frac{1}{\nu} \left(\mathbb{E}\left[\Delta_{0} + T \left[ \frac{\eta^2L}{2}\frac{Z_2}{m} + \frac{\eta \mu}{120}\ep \right]\right]\right) \\
    &= \frac{1}{\nu}\left(\Delta_0 + T \left[ \frac{\eta^2L}{2}\frac{Z_2}{m} + \frac{\eta\mu}{120}\ep \right]\right).
\end{align*}
}
Then, continuing with the calculations, choosing $\nu = 10J_0$, $\eta \leq \frac{\mu\ep m}{240 L Z_2}$ and $T = \frac{4}{\eta\mu} \log(120\Delta_0/\ep)$, it follows that
\begin{align*}
    \bbP\left(\max_{0 \leq t \leq T} Y_t \geq \nu \right) \leq \frac{1}{\nu} \left(\Delta_0 + \frac{\ep}{5}\log(120\Delta_0/\ep)\right) \leq \frac{1}{5},
\end{align*}
if we choose $\ep$ small enough such that $ \ep\log(120\Delta_0/\ep)) \leq 5\Delta_0$. 
Suppose $\tau \leq  T$. Then, $\max_t Y_t \geq Y_{\tau} \geq \Delta_{\tau} \geq 10 J_0$. Hence, 
\begin{align*}
    \bbP(\tau \leq T) &\leq \bbP\left(\max_{0 \leq t \leq T} Y_t \geq \nu \right) \leq \frac{1}{5}.
\end{align*}
This completes the proof of Proposition~\ref{proposition:stability-first-time}.
\end{proof}
We are now ready to state the proof of Theorem \ref{theorem:lqr_tracking_algorithm_convergence}.
\begin{proof}[Proof of Theorem \ref{theorem:lqr_tracking_algorithm_convergence}]
Observe that for any $0 \leq t \leq T$,
\small
\begin{align*}
    \mathbb{E}^t[\Delta_{t+1} 1_{\tau> t+1}] &\leq \mathbb{E}^t[\Delta_{t+1} 1_{\tau > t}] = \mathbb{E}^t[\Delta_{t+1}] 1_{\tau > t}.
\end{align*}
\normalsize
We will then bound $\mathbb{E}^t[\Delta_{t+1} 1_{\tau> t+1}]$ by bounding $\mathbb{E}^t[\Delta_{t+1}] 1_{\tau > t}$.
There are two cases to consider. 
\begin{enumerate}
    \item The first is when $\tau > t$. In this case, by Lemma \ref{lemma: per-iteration-change-in-optimality gap-first-time}, then
    \small
\begin{align*}
    \mathbb{E}^t[\Delta_{t+1}] \leq \left(1 - \frac{\eta\mu}{4}\right) \Delta_t + \frac{\eta^2L}{2}\left(\frac{Z_2}{m}  \right) + \frac{\eta \mu}{120}\ep.
\end{align*}
\normalsize
    \item The second is when $\tau \leq t$. Here,  $\mathbb{E}^t[\Delta_{t+1}] 1_{\tau > t} = 0$.
\end{enumerate}
Thus, combining the two cases, we obtain that
{\small
\begin{align*}
    \mathbb{E}^t[\Delta_{t+1}] 1_{\tau > t} \leq \left(1 - \frac{\eta\mu}{4}\right) \Delta_t 1_{\tau > t} + \frac{\eta^2L}{2}\left(\frac{Z_2}{m} \right) + \frac{\eta\mu}{120}\ep.
\end{align*}}
\noindent Taking expectations over $\mathcal{F}_t$ and using induction, we find
\small
\begin{align*}
    &\mathbb{E}^t[\Delta_{t+1}] 1_{\tau > t+1} \\
    &\leq \left(1 - \frac{\eta\mu}{4}\right)^{t+1} \Delta_0 + 
    \left(\frac{\eta^2 L}{2} \frac{Z_2}{m} +   \frac{\eta \mu}{120}\ep\right) \sum_{i=0}^t \left(1-\frac{\eta\mu}{4}\right)^i \\
    &\labelrel\leq{eq:inf_sum_1/p-main} \left(1 - \frac{\eta\mu}{4}\right)^{t+1} \Delta_0 +
    \frac{2\eta L}{\mu}\left(\frac{Z_2}{m} \right) + \frac{4\ep}{120},
\end{align*}
\normalsize
where \eqref{eq:inf_sum_1/p-main} is a consequence of $\sum_{i=0}^{\infty} \left(1 - \frac{\eta\mu}{4}\right)^i \leq \frac{1}{1 - (1 - \eta\mu/4)} = \frac{4}{\eta\mu}$.
After substituting $T = \frac{4}{\eta\mu} \log(\frac{120 \Delta_0}{\ep})$, since $\eta \leq \frac{\mu\ep m}{240 L Z_2}$, we find that 
    \begin{align*}
    \mathbb{E}[\Delta_T 1_{\tau > T}] &\leq \frac{\ep}{120} + \frac{\ep}{120} + \frac{\ep}{30} = \frac{\ep}{20}.
    \end{align*} 
Hence,
\small
\begin{align*}
    \bbP(\Delta_T \geq \ep) &\leq \bbP(\Delta_T 1_{\tau > T} \geq \ep) + \bbP(1_{\tau \leq T}) \\
    &\labelrel\leq{eq:markov+stability_prob-main} \frac{\mathbb{E}[\Delta_T 1_{\tau > T}]}{\ep} + \frac{1}{5} \leq \frac{1}{20} + \frac{1}{5} = \frac{1}{4}.
\end{align*}
\normalsize
We used Markov's inequality and Proposition \ref{proposition:stability-first-time} to get \eqref{eq:markov+stability_prob-main}.
\end{proof}
\vspace{-10pt}
\section{Numerical results}
\vspace{-5pt}
To demonstrate convergence, we test our algorithm on a LQR problem with A,B, Q, R matrices each in $\bbR^{3 \times 3}$, which we list below:
\begin{align*}
    &A = \left[ \begin{matrix} 1 & 0 & -10 \\ -1 & 1 & 0 \\ 0 & 0 & 1 \end{matrix} \right], \quad B = \left[ \begin{matrix} 1 & -10 & 0 \\ 0 & 1 & 0 \\ -1 & 0 & 1 \end{matrix} \right], \quad Q = \left[ \begin{matrix} 2 & -1 & 0 \\ -1 & 2 & -1 \\ 0 & -1 & 2 \end{matrix} \right], \quad R = \left[ \begin{matrix} 5 & -3 & 0 \\ -3 & 5 & -2 \\ 0 & -2 & 5 \end{matrix} \right].
\end{align*}
We sample the initial state $x_0$ from $N(0,I)$, and set the discount factor as $\gamma = 0.9$.  The tracking target $x^*$, fixed throughout, is randomly sampled from $N(0,I)$. In Figure \ref{fig:minibatch_plot}, we show the learning process of our algorithm for three settings of the minibatch size $m$. For each $m$, we tuned the stepsize $\eta$ to be the maximum possible ensuring convergence. We see that within 2000 iterations, our proposed algorithm can converge to almost $10^{-1}$ for $m = 1$ and almost $10^{-2}$ for $m = 5, 10$. By raising the minibatch size from $m = 1$ to $m = 5$, we could choose a stepsize that was 5 times larger, and this reduced optimality gap by almost 10 times after 2000 iterations. However, further increasing $m$ from 5 to 10 yields no convergence speedup, as we are unable to further increase the stepsize $\eta$ without leading to divergence. This reflects our earlier observation that  {\small$\eta \leq \tilde{O}\left(\frac{1}{L}, \frac{\rho\mu}{d\lambda}\right)$} has to hold.
\begin{figure}[H]
    \centering
    \includegraphics[width = 0.5\textwidth]{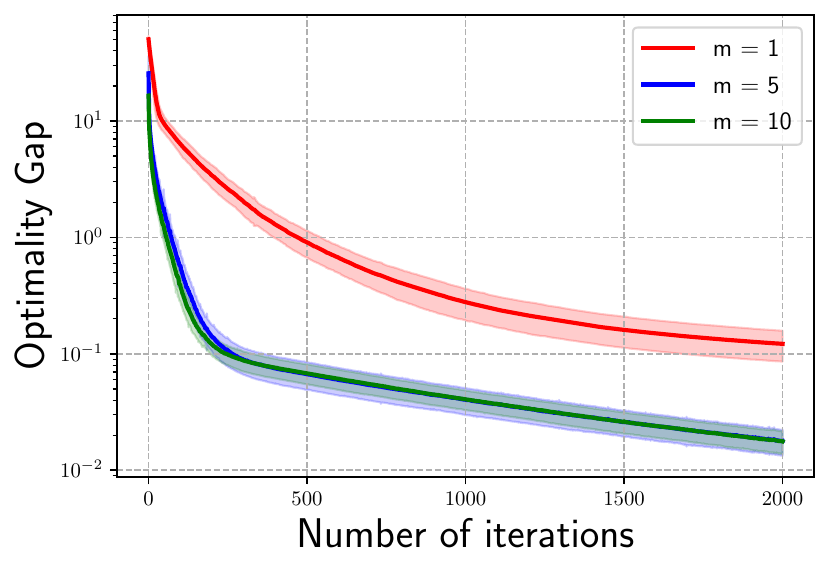}
    \caption{Optimality gap (mean trajectory and 1 std-dev confidence band over 50 trials)}
    \label{fig:minibatch_plot}
\end{figure}


\vspace{-5pt}
\section{Conclusion and future work}
In this paper, we proved that gradient dominance and local smoothness hold for the general LQR tracking problem. This paved the way for the development of a zeroth-order algorithm that can converge to the globally optimal policy. We highlight a few directions for future work: (1) considering distributed settings with heterogeneous dynamics, (2) studying more general cost functions, (3) studying adversarial noise. 
\vspace{-3pt}

\bibliographystyle{ieeetr}
\bibliography{references}

\begin{appendices}

\section{Optimal controller for LQR tracking problem}

We derive in this section the optimal controller for the cost $J$ in Equation \eqref{eq:LQR-tracking}. We will show that the optimal controller takes the form $u_t = Kx_t + \tilde{K}x^*$.

\begin{proposition}[Optimal controller for $J$]
\label{proposition:optimal controller for J}
Suppose the system given by $(A, B)$ is controllable. The cost-to-go value function, $V(x_0)$, for the discounted infinite-horizon problem \ref{eq:LQR-tracking}, takes the form
\begin{align*}
    V(x_0) = x_0^\top P x_0 + 2x_0^\top q + r,
\end{align*}
where 
\begin{align*}
    &P = Q - (\gamma B^\top PA)^\top (R + \gamma B^\top PB)^{-1} (\gamma B^\top PA) + \gamma A^\top PA, \\
    &q = -(I - \gamma (A+BK)^\top)^{-1}Qx^*, \\
    &r = \frac{1}{1-\gamma} \left( (x^*)^\top Qx^* - (\gamma B^\top q)^\top (R + \gamma B^\top P B)^{-1} (\gamma B^\top q) \right).
\end{align*}
Moreover, the optimal controller is time-invariant and of the form
\begin{align*}
    u = Kx + \tilde{K}x^*, where
\end{align*}
\[K = - (R + \gamma B^\top P B)^{-1} (\gamma B^\top PA), \quad \tilde{K} =-(R + \gamma B^\top P B)^{-1}\gamma B^\top (I - \gamma (A+BK)^\top)^{-1}Q.\]
Then, this same controller is also the optimal policy to minimize the expected cost $J$ (where expectation is taken over the initial state $x_0$). The optimal expected cost takes the following form
\begin{align*}
    J^{\mathrm{opt}} &= \mathbb{E}[x_0^\top Px_0 + 2x_0^\top q + r] = \tr(P\Sigma) + 2 \mathbb{E}[x_0]^\top q + r,
\end{align*}
where $\Sigma = \mathbb{E}[x_0x_0^\top]$.
\end{proposition}

\begin{proof}
We note that the cost-to-go value function starting from $x_0$ takes the form 
\[V(x_0) = x_0^\top P x_0 + 2x_0^\top q + r, \]
which is motivated by appealing to the value function in the finite horizon \cite{anderson2007optimal} and taking the limit as $t$ goes to infinity. 

Then, 
\begin{align*}
    V(x) = \min_u (x - x^*)^\top Q(x - x^*) + u^\top Ru + \gamma V(Ax + Bu). 
\end{align*}
Taking the gradient with respect to $u$, we find that
\begin{align*}
    \nabla_u &= 2Ru + \nabla_u (\gamma \left((Ax + Bu)^\top P(Ax + Bu) + 2(Ax + Bu)^\top q + r\right)) \\
    &= 2Ru + 2\gamma B^\top P(Ax + Bu) + 2\gamma B^\top q.
\end{align*}
Setting $\nabla_u$ to be 0, we find that
$2(R + \gamma B^\top P B)u + 2\gamma B^\top PAx + 2\gamma B^\top q = 0$, which implies that
\begin{align*}
    u^* = - (R + \gamma B^\top P B)^{-1} (\gamma B^\top PA x + \gamma B^\top q).
\end{align*}
Then, plugging this back into $V(x)$,
\begin{align*}
    V(x) &= x^\top Q x + 2x^\top(-Qx^*) + (x^*)^\top Qx^* \\
    &\quad + \left((R + \gamma B^\top P B)^{-1}(\gamma B^\top PAx + \gamma B^\top q) \right)^\top R \left((R + \gamma B^\top P B)^{-1}(\gamma B^\top PAx + \gamma B^\top q) \right) \\
    &\quad + \gamma [ (Ax - B \left((R + \gamma B^\top P B)^{-1}(\gamma B^\top PAx + \gamma B^\top q) \right))^\top P \\
    &\quad \quad (Ax - B \left((R + \gamma B^\top P B)^{-1}(\gamma B^\top PAx + \gamma B^\top q) \right))^\top] \\
    &\quad + \gamma \left(2q^\top (Ax - B \left((R + \gamma B^\top P B)^{-1}(\gamma B^\top PAx + \gamma B^\top q) \right))\right) + \gamma r \\
    &= x^\top Qx + 2x^\top(-Qx^*) + (x^*)^\top Qx^* \\
    &\quad + (\gamma B^\top P Ax + \gamma B^\top q)^\top (R + \gamma B^\top PB)^{-1}(\gamma B^\top P Ax + \gamma B^\top q) \\
    &\quad - \gamma \left[2(Ax)^\top P(B\left((R + \gamma B^\top P B)^{-1}(\gamma B^\top PAx + \gamma B^\top q) \right)) \right] \\
    &\quad + \gamma x^\top A^\top PAx \\
    &\quad + 2\gamma \left[q^\top Ax - q^\top (B\left((R + \gamma B^\top P B)^{-1}(\gamma B^\top PAx + \gamma B^\top q) \right)) \right] + \gamma r \\
    &= x^\top Qx + 2x^\top (-Qx^*) + (x^*)^\top Qx^* \\
    &\quad - (\gamma B^\top PAx + \gamma B^\top q)^\top (R + \gamma B^\top PB)^{-1}(\gamma B^\top PAx + \gamma B^\top q) \\
    &\quad + \gamma x^\top A^\top PAx + 2\gamma q^\top Ax + \gamma r
\end{align*}
Collecting terms, we find that 
\begin{align*}
    &P = Q - (\gamma B^\top PA)^\top (R + \gamma B^\top PB)^{-1} (\gamma B^\top PA) + \gamma A^\top PA, \\
    &q = -Qx^* - (\gamma B^\top PA)^\top(R + \gamma B^\top PB)^{-1} (\gamma B^\top q) + \gamma A^\top q, \\
    &r = \frac{1}{1-\gamma} \left((x^*)^\top Qx^* - (\gamma B^\top q)^\top (R + \gamma B^\top P B)^{-1} (\gamma B^\top q) \right).
\end{align*}
Finally, observe that after rearranging the equation involving $q$, 
\begin{align*}
    q = -(I - \gamma (A+BK)^\top)^{-1}Qx^*.
\end{align*}
Plugging this back into the optimal formula for $u^*$ then yields our desired result for the case with fixed $x_0$. Since the optimal controller is independent of the initial state $x_0$, and the stochasticity in the population cost $J$ is over the initial state $x_0$, this shows that the same controller is optimal in minimizing $J$.
\end{proof}

\section{Properties of LQR tracking}

\subsection{Non-convexity}
\label{appendix:non-convexity}
We provide here a simple example, $n = 2$ ( here $n$ is state dimension), showing that $J(K)$, the zero-target LQR cost, is non-convex. By extension $J(K,\tilde{K})$ cannot be convex in general as well.
\begin{proposition}
Consider the zero-target LQR cost
\[J(K) = \sum_{t=0}^{\infty} x_t^\top Q x_t + u_t^\top R u_t,\]
where $u_t = Kx_t$, and $x_{t+1} = Ax_t + Bu_t$, where $x_t \in \bbR^n$ and $u_t \in \bbR^k$. Then $J(K)$ is in general non-convex for $n \geq 2$.
\end{proposition}
\begin{proof}
We provide a counterexample where the set of stable controllers $K$ is not convex. Pick
\[A = \begin{bmatrix}
1 & 0 \\
0 & 1
\end{bmatrix}, \quad B = \begin{bmatrix}
1 & 0 \\
0 & 1
\end{bmatrix}, \quad K_1 = \begin{bmatrix}
-1.5 & 4042 \\
0 & -1.5
\end{bmatrix}, \quad
K_2 = \begin{bmatrix}
-1.5 & 0\\
4040 & -1.5
\end{bmatrix}.
\]
Then, $(A,B)$ is a controllable system. In addition, observe that
\[ 
\rho (A+BK_1) = \rho\left(\begin{bmatrix}
-0.5 & 4042 \\
0 & -0.5
\end{bmatrix} \right) < 1, \quad \rho (A+BK_2) = \rho\left(\begin{bmatrix}
-0.5 & 0 \\
4042 & -0.5
\end{bmatrix} \right) < 1.
\]
However,
\[ 
\rho\left(A + B\left(\frac{K_1 + K_2}{2} \right)\right) = 
\rho\left(\begin{bmatrix}
-0.5 & 2021 \\
2021 & -0.5
\end{bmatrix}  \right) > 1,
\]
which can be shown via direct computation. This proves that LQR is in general non-convex for $n \geq 2$.
\end{proof}

\subsection{Gradient dominance} \label{appendix:gradient-dominance}
In this subsection, we fill in the gaps in our earlier proof sketch of Proposition \ref{proposition:gradient dominance of lqr tracking}. We denote $g := \tilde{K}x^*$. For notational convenience, we interchangeably refer to a policy given by $u_t = Kx_t + \tilde{K}x^*$ as $(K,g)$ or $(K,\tilde{K})$. The corresponding cost-to-go-function is  $V_{K,\tilde{K}} := V_{K,g}: \bbR^n \to \bbR_{\geq 0}$, where for any $z \in \bbR^n$,
\begin{align*}
    & V_{K,\tilde{K}}(z) := V_{K,g} (z) =  \sum_{t = 0}^{\infty} \gamma^t c(x_t, Kx_t+ g), \quad \mbox{where } x_0 = z.
\end{align*}
We now state and prove a lemma describing the value function of $J$ that we required in the proof sketch of Proposition \ref{proposition:gradient dominance of lqr tracking}.
\begin{lemma}[Value function of $J$]
\label{lemma:value function of J(K,g)}
Suppose $\sqrt{\gamma} \rho((A + BK))< 1$. Define $g := \tilde{K}x^*$. Then,
\begin{align*}
    V_{K,\tilde{K}}(z) := V_{K,g}(z) = z^\top P_K z + 2z^\top q_{K,g} + r_{K,g}, 
\end{align*}
where 
\begin{align*}
    &P_K = Q + K^\top RK + \gamma (A+BK)^\top P_K(A+BK), \\
    &q_{K,g} = (I - \gamma(A+BK)^\top)^{-1}\left( -Qx^* + K^\top Rg + \gamma (A+BK)^\top P_K Bg \right), \\
    &r_{K,g} = \frac{1}{1-\gamma} \left((x^*)^\top Qx^* + g^\top Rg + \gamma \left(g^\top B^\top P_K Bg
+ 2g^\top B^\top q_{K,g} \right) \right),
\end{align*}
and we can also rewrite $q_{K,g}$ as $q_{K,g} := q_{K,\tilde{K}}x^*$, where
\begin{align*}
    q_{K,\tilde{K}} = (I - \gamma(A+BK)^\top)^{-1}\left( -Q + K^\top R\tilde{K} + \gamma (A+BK)^\top P_K B\tilde{K} \right).
\end{align*}
\end{lemma}
\begin{proof}
We note that direct calculations show that $V_{K,g}(z) = z^\top P_{K}z + 2z^\top q_{K,g} + r$ for some $P_K, q_{K,g},$ and $r_{K,g}$. We will seek to find a recursive formula for $P_K, q_{K,g}$ and $r_{K,g}$. Since $V_{K,g}(z) = z^\top P_K z + 2z^\top q_{K,g} + r_{K,g}$, 
\begin{align*}
    V_{K,g}(z) &= z^\top P_K z + 2z^\top q_{K,g} + r_{K,g} \\
    &= c(z, Kz + g) + \gamma V_{K,g}(Az + B(Kz + g) ) \\
    &= (z - x^*)^\top Q(z - x^*) + (Kz +g)^\top R (Kz + g) \\
    &\quad + \gamma (A +BK)z + Bg)^\top P_K((A+ BK)z + Bg) + 2((A+BK)z + Bg )^\top q_{K,g} + r_{K,g} \\
    &= z^\top (Q + K^\top RK + (A+BK)^\top P_K(A+BK))z \\
    &\quad + 2z^\top (-Qx^* + K^\top Rg + (A+BK)^\top P_K Bg + (A+BK)^\top q_{K,g}) \\
    &\quad + (x^*)^\top x^* + g^\top R g + \gamma g^\top B^\top P_K Bg + 2g^\top B^\top q_{K,g} + r_{K,g}.
\end{align*}
Matching coeffients, we find that
\begin{align*}
    &P_K = Q + K^\top RK + \gamma (A+BK)^\top P_K(A+BK), \\
    &q_{K,g} = -Qx^* + K^\top Rg + \gamma (A+BK)^\top P_K Bg + \gamma (A+BK)^\top q_{K,g}, \\
    &r_{K,g} = (x^*)^\top Qx^* + g^\top R g + \gamma \left[g^\top B^\top P_K Bg + 2g^\top B^\top q_{K,g} + r_{K,g} \right].
\end{align*}
Some algebraic simplifications for the equations involving $q_{K,g}$ and $r_{K,g}$ then yield the result.
\end{proof}

Next we compute $\nabla J$, which we also used in proving Proposition \ref{proposition:gradient dominance of lqr tracking}.
\begin{lemma}
\label{lemma:gradient of J}
Define $$C_{K} = RK + \gamma B^\top P_K(BK + A), \quad d_{K,\tilde{K}} = R\tilde{K}+ \gamma B^\top P_KB\tilde{K}+ \gamma B^\top q_{K,\tilde{K}}.$$
The gradients of $J$ with respect to $K$ and $\tilde{K}$ are the following:
\begin{align*}
    &\nabla_K J  = 2C_K\Sigma_{K,\tilde{K}} + 2d_{K,\tilde{K}}x^* \rho_{K,\tilde{K}}^\top,\quad \nabla_{\tilde{K} } J = \left(2C_K\rho_{K,\tilde{K}} + 2\beta_{\gamma}d_{K,\tilde{K}}x^* \right)(x^*)^\top,
\end{align*}
where
\begin{align*}
    \Sigma_{K,\tilde{K}} := \mathbb{E} \left[\sum_{t=0}^{\infty} \gamma^t x_t x_t^\top\right], \quad \rho_{K,\tilde{K}} := \mathbb{E}\left[\sum_{t=0}^{\infty} \gamma^t x_t \right], \quad \beta_{\gamma} := \sum_{t=0}^{\infty} \gamma^t. 
\end{align*}
\end{lemma}
\begin{proof}
Observe that for any $z \in \bbR^n$, 
\begin{align*}
    V_{K,g}(z) &= c(z, Kz + g) + \gamma V((A+BK)z + Bg) \\
    &= (z - x^*)^\top Q(z - x^*) + (Kz + g)^\top R(Kz + g) \\
    & \quad + \gamma \left(((A+BK)z + Bg)^\top P_{K}((A+BK)z + Bg) + 2((A+BK)z + Bg)^\top q_{K,g} + r_{K,g}\right).
\end{align*}
Taking the gradient first with respect to $K$, using the product rule we find that 
\begin{align*}
    \nabla_K V_{K,g}(z) &= 2R(Kz + g) z^\top + 2\gamma \left(B^\top P_{K}((A+BK)z + Bg)z^\top +  B^\top q_{K,g}z^\top \right) \\ & \quad + \gamma \nabla_K V_{K,g}(x_1)\big|_{x_1 = (A+BK)z + Bg}.
\end{align*}
Using recursion, and taking expectations, we find that 
\begin{align*}
    \nabla_K J &= \mathbb{E} \left[\sum_{t=0}^{\infty} 2\gamma^t \left((RK + \gamma B^\top P_{K}(A+BK)) x_t x_t^\top + (Rg + \gamma B^\top P_{K}Bg + B^\top q_{K,g})x_t^\top \right)\right]
\end{align*}
We next compute the gradient with respect to $g$. Observe that
\begin{align*}
    \nabla_g V_{K,g}(z) &= 2R(Kz + g) + 2\gamma B^\top P_K((A+BK)z + Bg) + 2\gamma B^\top q_{K,g} +\gamma \nabla_g V(x_1) \big|_{x_1 = (A+BK)z + Bg}.
\end{align*}
Using recursion, and taking expectations, we find that
\begin{align*}
    \nabla_g J &= 2 \mathbb{E} \left[\sum_{t=0}^{\infty} \gamma^t \left( (RK + \gamma B^\top P(A+BK))x_t + Rg + \gamma B^\top P_K Bg + \gamma B^\top q_{K,g} \right)\right].
\end{align*}

Since $g = \tilde{K}x^*$, it follows then that 
\begin{align*}
    \nabla_{\tilde{K} } J &= 2 \mathbb{E} \left[\sum_{t=0}^{\infty} \gamma^t \left( (RK + \gamma B^\top P(A+BK))x_t + Rg + \gamma B^\top P_K Bg + \gamma B^\top q_{K,g} \right)\right](x^*)^\top.
\end{align*}

\end{proof}

We define the advantage function $A_{K,g}(x, u)$ as
\begin{align*}
    A_{K,g}(x,u) := c(x,u) + \gamma V_{K,g}(Ax + Bu ) - V_{K,g}(x).
\end{align*}

For any two policies $(K,g)$ and $(K',g')$,  the next lemma provides an expression for the difference between $J(K,g)$ and $J(K',g')$ in terms of the advantage function. This played a key role in the proof sketch of Proposition \ref{proposition:gradient dominance of lqr tracking}.
\begin{lemma}
\label{lemma:cost-difference}
Note that any two policies $(K, \tilde{K})$ and $(K', \tilde{K}')$ correspond to policies $(K,g)$ and $(K,g')$, where $u_t = Kx_t +g$ and $u_t' = K'x_t' + g'$, and $g =\tilde{K}x^*, g' = \tilde{K}'x^*$. Let $\{(x_t,u_t)\}$ and $\{(x_t',u_t')\}$ denote the state-action trajectories corresponding to $(K,g)$ and $(K',g')$ respectively. Then, overloading notation to so that $J(K,\tilde{K}) := J(K,g)$,
\begin{align*}
    J(K,\tilde{K}) - J(K',\tilde{K}') = J(K,g) - J(K',g') &= - \mathbb{E} \sum_{t=0}^{\infty} A_{K,g}(x_t', K'x_t' + g').
\end{align*}
Moreover, defining
\begin{align*}
    &C_K = RK + \gamma B^\top P(BK + A), \quad d_{K,\tilde{K}} = R\tilde{K}+ \gamma B^\top PB\tilde{K} + \gamma B^\top q_{K,\tilde{K}}, \\
    & \Sigma^{\mathrm{opt}}  = \mathbb{E} \left[\sum_{t=0}^{\infty} \gamma^t x_t^{\mathrm{opt}} (x_t^{\mathrm{opt}})^\top \right], \quad \rho^{\mathrm{opt}} = \mathbb{E}\left[\sum_{t=0}\gamma^t x_t^{\mathrm{opt}} \right], \quad \beta_\gamma = \sum_{t=0}^{\infty} \gamma^t,
\end{align*}
for an optimal policy $(K^{\mathrm{opt}}, \tilde{K}^{\mathrm{opt}})$, we have
\begin{align*}
    &J(K,\tilde{K}) - J(K^{\mathrm{opt}}, \tilde{K}^{\mathrm{opt}}) \leq \frac{\Psi^{\mathrm{opt}}}{\sigma_{\min}(R)} \left( \norm{C_K}^2 + \norm{d_{K,\tilde{K}}x^*}^2 \right),
\end{align*}
where 
\begin{align*}
\Psi^{\mathrm{opt}} :=  \max\left\{\norm{\Sigma^{\mathrm{opt}}}_2 + \norm{\rho^{\mathrm{opt}}}_2, \sum_{t=0}^{\infty}\gamma^t + \norm{\rho^{\mathrm{opt}}}_2\right\}.
\end{align*}
\end{lemma}

\begin{proof}

First consider the difference in cost-to-go, starting from the initial state $x_0$. 
\begin{align*}
    V_{K',g'}(x_0) - V_{K,g}(x_0) &=  \sum_{t=0}^{\infty} \gamma^t c(x_t',u_t') - V_{K,g}(x_0) \\
    &= \sum_{t=0}^{\infty} \gamma^t \left(c(x_t', u_t') + V_{K,g}(x_t') - V_{K,g}(x_t')\right) - V_{K,g}(x_0) \\
    &=  \sum_{t=0}^{\infty} \gamma^t \left(c(x_t',u_t') + \gamma V_{K,g}(x_{t+1}') - V_{K,g}(x_t') \right).
\end{align*}
For the last line, we utilized the fact that $x_0 = x_0'$.

Then, we get that
\begin{align*}
    V_{K',g'}(x_0) - V_{K,g}(x_0) &= \sum_{t = 0}^{\infty} \gamma^t A_{K,g}(x_t', u_t') \\
\end{align*}
We next compute $A_{K,g}(x, K'x+g')$. Recall that 
\begin{align}
    V_{K,g}(x) &= (x - x^*)^\top Q(x - x^*) + (Kx + g)^\top R(Kx+ g) + \gamma V_{K,g}((A+BK)x + Bg).
    \label{eq:V_Kg_recursive_expression}
\end{align}
Then, observe that 
\begin{align*}
    &A_{K,g}(x, K'x + g') \\
    &= c(x,K'x + g') + \gamma V_{K,g}(Ax + B(K' x+ g') ) - V_{K,g}(x) \\
    &= (x - x^*)^\top Q(x - x^*) + (K'x + g')^\top R(K' x+ g') \\
    & \quad + \gamma V_{K,g}(Ax + B(K' x+ g')) - V_{K,g}(x) \\
    &= x^\top (Q + (K')^\top RK')x + 2x^\top \left( -Qx^* + (K')^\top Rg' \right) + g'^\top R g' + (x^*)^\top Qx^*\\
    & \quad + \gamma \left((A+BK')x + Bg' \right)^\top P_K \left((A+BK')x + Bg' \right) \\
    & \quad + 2 \gamma \left( (A+BK')x + Bg' \right)^\top q_{K,g}  + \gamma r_{K,g}  - V_{K,g}(x) \\
    &= x^\top \left(Q + (K + K' - K) ^\top R(K + K' - K)\right) x  + (x^*)^\top Qx^*\\
    &\quad + 2x^\top \left(-Qx^* + (K + K' - K)^\top R (g + g' - g) \right) + (g + g' - g)^\top R (g + g' - g) \\
    &\quad + \gamma \left((A + B(K + K'-K))x + B(g + g' - g) \right)^\top P_K \left((A + B(K + K'-K))x + B(g + g' - g) \right)\\
    &\quad + 2\gamma \left((A +B(K + K'- K))x + B(g + g'-g) \right)^\top q_{K,g} + \gamma r_{K,g} - V_{K,g}(x) \\
    &\labelrel={eq:plug_in_Vkg} x^\top \left(2 (K' - K)^\top R K + 
    (K' - K)^\top R (K' - K)\right)x + 2x^\top \left((K' - K)^\top Rg' + K^\top R(g' - g)\right) \\
    & \quad +2\left( (g'-g)^\top R g\right) + (g' - g)^\top R(g' - g) \\ 
    &\quad + \gamma \left[ 2x^\top\left(B(K'-K)\right)^\top P_K(A+BK)x + x^\top \left( B(K'-K)\right)^\top P_K \left(B(K'-K)\right)x\right] \\
    & \quad + \gamma \left[2 x^\top \left(B(K'-K)\right)^\top P_K(Bg') + 2x^\top (A+BK)^\top P_K(B(g' - g))\right] \\
    & \quad + \gamma \left[2x^\top(B(K'- K))^\top q_{K,g}\right] \\
    &\quad + \gamma \left[2(B(g'-g))^\top P_K(Bg) + (B(g'-g))^\top P_K(B(g'-g)) + 2 (B(g'-g))^\top q_{K,g}\right] \\
    &= 2 \tr (xx^\top (K'- K)^\top RK) + 2 \tr (x^\top (K'-K)^\top R(g'-g)) + 2 \tr (x^\top (K' - K)^\top Rg ) \\
    & \quad + 2 \gamma \tr (x^\top (K' - K)^\top B^\top P_K (BK + A)x) + 2 \gamma \tr (x^\top (K' - K)^\top B^\top P_K B(g'-g))\\
    & \quad + 2\gamma \tr(x^\top (K' - K)^\top B^\top P_K Bg) + 2\gamma \tr (x^\top (K' - K)^\top B^\top q_{K,g}) \\
    & \quad + \tr (xx^\top (K' - K)^\top (R + \gamma B^\top P_K B) (K'- K)) \\
    & \quad + 2 \tr (x^\top K^\top R (g' - g)) + 2 \tr (g^\top R (g' - g)) + 2 \gamma \tr((g' - g)^\top B^\top P_K(A +BK)x) \\
    & \quad + 2 \gamma \tr ((g' - g)^\top B^\top P_K B g) + 2 \gamma \tr ((g' - g)^\top B^\top q_{K,g}) \\
    & \quad + (g' - g)^\top (R + \gamma B^\top P_K B)(g' - g) \\
    &= 2 \tr((K'-K)^\top (RK + B^\top P_K(BK+A))xx^\top) \\
    &\quad + 2\tr((K'-K)^\top (Rg + \gamma B^\top P_K Bg + \gamma B^\top q_{K,g})x^\top) \\
    &\quad + 2\tr((g'-g)^\top ((RK + \gamma B^\top P_K (A+BK))x + Rg \gamma B^\top P_KBg + \gamma B^\top q_{K,g}) \\
    &\quad + 2 \tr((K'-K)^\top (R + \gamma B^\top P_K B)(g'-g) x^\top) \\
    &\quad + \tr(xx^\top(K'-K)^\top (R + \gamma B^\top P_K B)(K'-K)) + (g'-g)^\top (R + \gamma B^\top P_K B)(g'-g) \\
    &= 2 \tr((K'-K)^\top (C_K x + d_{K,\tilde{K}}x^*)x^\top) + 2\tr((g'-g)^\top (C_K x + d_{K,\tilde{K}}x^*)) \\
    &\quad + 2 \tr((K'-K)^\top (R + \gamma B^\top P_K B)(g'-g) x^\top) \\
    &\quad + \tr(xx^\top(K'-K)^\top (R + \gamma B^\top P_K B)(K'-K)) + (g'-g)^\top (R + \gamma B^\top P_K B)(g'-g)  
\end{align*}
Above, \eqref{eq:plug_in_Vkg} holds by plugging in the expression for $V_{K,g}(x)$ in Line \eqref{eq:V_Kg_recursive_expression}.
For the last line, we defined 
$$C_K := RK + \gamma B^\top P_K (BK + A), \quad d_{K,\tilde{K}} := R\tilde{K} + \gamma B^\top P_K B\tilde{K} + \gamma B^\top q_{K,\tilde{K}}.$$

Completing the square, using the fact that $(R + \gamma B^\top P_K B)$ is positive definite, we find that the advantage function satisfies the following inequality: 
\begin{align}
    &A_{K,g}(x, K'x + g') \nonumber \\
    &= \left((K' - K)x +  (g' - g) + (R+\gamma B^\top P_K B)^{-1} \left(C_Kx + d_{K,\tilde{K}}x^*\right) \right)^\top(R + \gamma B^\top P_K B) \nonumber \\
    & \quad \left((K' - K)x +  (g' - g) + (R+\gamma B^\top P_K B)^{-1} \left(C_Kx + d_{K,\tilde{K}}x^*\right) \right) \nonumber \\
    & \quad - (C_Kx + d_{K,\tilde{K}}x^*)^\top (R + \gamma B^\top P_K B)^{-1} (C_Kx + d_{K,\tilde{K}}x^*) \label{eq:A_Kg_equality}\\
    &\geq - (C_Kx + d_{K,\tilde{K}}x^*)^\top (R + \gamma B^\top P_K B)^{-1} (C_Kx + d_{K,\tilde{K}}x^*) \label{eq:A_Kg_inequality}
\end{align}
Using \eqref{eq:A_Kg_inequality}, and taking expectations, we find that 
\begin{align*}
    J(K',\tilde{K}') - J(K,\tilde{K}) &\geq - \mathbb{E}\left[\sum_{t=0}^{\infty} \gamma^t \left((C_Kx_t' + d_{K,\tilde{K}}x^*)^\top (R + \gamma B^\top P_K B)^{-1} (C_Kx_t' + d_{K,\tilde{K}}x^*)\right)\right].
\end{align*}
Applying this to an optimal policy $(K^{\mathrm{opt}}, \tilde{K}^{\mathrm{opt}})$, we find that
\small
\begin{align*}
    & J(K,\tilde{K}) - J(K^{\mathrm{opt}}, \tilde{K}^{\mathrm{opt}})\\
    &\leq \mathbb{E}\left[\sum_{t=0}^{\infty} \gamma^t \left((C_Kx_t^{\mathrm{opt}} + d_{K,\tilde{K}}x^*)^\top (R + \gamma B^\top P_K B)^{-1} (C_Kx_t^{\mathrm{opt}} + d_{K,\tilde{K}}x^*)\right)\right] \\
    &\leq \frac{1}{\sigma_{\min}(R)} \mathbb{E}\left[\sum_{t=0}^{\infty} \gamma^t \left((C_Kx_t^{\mathrm{opt}} + d_{K,\tilde{K}}x^*)^\top (C_Kx_t^{\mathrm{opt}} + d_{K,\tilde{K}}x^*)\right)\right] \\
    &= \frac{1}{\sigma_{\min}(R)} \left( \tr\left(C_K^\top C_K\mathbb{E}\left[\sum_{t=0}^{\infty} \gamma^t x_t^{\mathrm{opt}}(x_t^{\mathrm{opt}})^\top\right]\right) + 2 \tr\left(\left(\mathbb{E} \left[\sum_{t=0}^{\infty} \gamma^t x_t^{\mathrm{opt}}\right]\right)^\top C_K^\top d_{K,\tilde{K}}x^*\right) + \sum_{t=0}^{\infty} \gamma^t \norm*{d_{K,\tilde{K}}x^*}^2 \right).
\end{align*}
\normalsize
Upon some algebraic manipulation we can then obtain the final result.
\end{proof}

We can now provide the full statement of the gradient dominance of the LQR  tracking problem. 
\begin{proposition}[Gradient dominance of LQR tracking]
\label{proposition: gradient dominant single agent}
Define 
$$\Psi^{\mathrm{opt}} =  \max\left\{\norm{\Sigma^{\mathrm{opt}}}_2 + \norm{\rho^{\mathrm{opt}}}_2, \sum_{t=0}^{\infty}\gamma^t + \norm{\rho^{\mathrm{opt}}}_2\right\}, $$
and
$$\Sigma^{\mathrm{opt}} := \mathbb{E} \sum_{t=0}^{\infty} \gamma^t (x_t)^{\mathrm{opt}}((x_t)^{\mathrm{opt}})^\top, \quad
\rho^{\mathrm{opt}} :=  \mathbb{E} \sum_{t=0}^{\infty} \gamma^t (x_t)^{\mathrm{opt}}
$$
are the discounted state covariance matrix and mean vector respectively corresponding to an optimal policy for $J$, and $\{x_t^{\mathrm{opt}}\}$ denotes the trajectory followed by the optimal policy.
Suppose 
$$\mathbb{E}[x_0] = 0, \quad \mathbb{E}[x_0 x_0^\top] = \Sigma \succeq \alpha I  \mbox{ for some } \alpha > 0, \quad  \abs{\rho(A+BK)} < 1.$$
There are two cases to consider. The first is when $x^* = 0$. In this case,
\begin{align*}
    J(K, \tilde{K}) - J^{\mathrm{opt}} \leq \frac{\Psi^{\mathrm{opt}}}{4\alpha^2 \sigma_{\min}(R)} \norm*{\nabla J(K, \tilde{K})}_F^2.
\end{align*}
The second is when $x^* \neq 0$. Then, 
\begin{align*}
    J(K, \tilde{K}) - J^{\mathrm{opt}} \leq \frac{\Psi^{\mathrm{opt}}}{4\min\{\alpha^2, \norm*{x^*}^2\} \sigma_{\min}(R)} \norm*{\nabla J(K, \tilde{K})}_F^2.
\end{align*} 
Moreover, we have the following lower bound
\begin{align*}
    J(K,\tilde{K}) - J^{\mathrm{opt}} \geq \frac{\min\{\alpha,1\}}{\norm{R+B^\top P_K B}_2} \left(\norm{C_K}_F^2 + \norm{d_{K,\tilde{K}}x^*}_F^2 \right).
\end{align*}

\end{proposition}
\begin{proof}
Our earlier proofs of Lemmas \ref{lemma:value function of J(K,g)}, \ref{lemma:gradient of J} and \ref{lemma:cost-difference} filled in the gaps in our earlier proof sketch of Proposition \ref{proposition:gradient dominance of lqr tracking}. It remains for us to prove the lower bound. Consider picking $$K' = K - (R + \gamma B^\top P_K B)^{-1} C_K, \quad \tilde{K}' = \tilde{K} - (R+\gamma B^\top P_K B)^{-1} d_{K,\tilde{K}}.$$ For notational simplicity, we introduce the following definitions:
\begin{align*}
    E_{K,\tilde{K}} = \begin{bmatrix}
    C_K & d_{K,\tilde{K}}x^*
    \end{bmatrix}, \quad N_{K,\tilde{K}} = \sum_{t=0}^{\infty} \gamma^t \begin{bmatrix}
    \bbE[(x_t') (x_t')^\top] & \bbE[x_t'] \\
    \bbE[(x_t')^\top] & 1
    \end{bmatrix}.
\end{align*}
Then, since $J(K,\tilde{K}) \geq J^{\mathrm{opt}}$, it follows by the equality in \eqref{eq:A_Kg_equality} and our choice of $K',\tilde{K}'$ that
\begin{align*}
    J(K,\tilde{K}) - J^{\mathrm{opt}} &\geq J(K,\tilde{K}) - J(K',\tilde{K}') \\
    &= -\bbE\left[\sum_{t=0}^{\infty} A_{K,g}(x_t', K'x_t' + \tilde{K}'x^*) \right] \\
    &= \bbE\left[\sum_{t=0}^{\infty} \gamma^t (C_K x_t' + d_{K,\tilde{K}}x^*)^\top (R + \gamma B^\top P_K B)^{-1} (C_K x_t' + d_{K,\tilde{K}}x^*) \right] \\
    &= \tr \left(E_{K,\tilde{K}}^\top (R + \gamma B^\top P_K B)^{-1} E_{K,\tilde{K}} N _{K,\tilde{K}}) \right) \\
    &\geq \frac{\min\{\alpha,1\}}{\norm{R+B^\top P_K B}_2}\tr(E_{K,\tilde{K}}^\top E_{K,\tilde{K}}) \\
    &= \frac{\min\{\alpha,1\}}{\norm{R+B^\top P_K B}_2} \left(\norm{C_K}_F^2 + \norm{d_{K,\tilde{K}}x^*}_F^2 \right).
\end{align*}

In obtaining the final inequality, we used the fact that 
\begin{align*}
    N_{K,\tilde{K}} \succeq \begin{bmatrix}
    \bbE[x_0 x_0^\top] & 0 \\
    0 & 1
    \end{bmatrix} \succeq \min\{\alpha,1\}I,
\end{align*}
where the first inequality follows from the same argument we had used for the $M_{K,\tilde{K}}$ term in the proof sketch of Proposition \ref{proposition:gradient dominance of lqr tracking}, and the second comes from our assumption that $\bbE[x_0 x_0^\top] \succeq \min\{\alpha,1\}I.$
\end{proof}

\subsection{Local smoothness} \label{appendix:local-smoothness}

\subsubsection{Outline}

Suppose $(K,\tilde{K})$ is in a sublevel set $\mathcal{G}_C$ of $J$ for some $C > 0$, and let $(K',\tilde{K}')$ denote another policy. We start with the following observation on the cost difference between $(K, \tilde{K})$ and $(K',\tilde{K}')$ for a single agent LQR tracking problem, using our computations from Lemma \ref{lemma:cost-difference}. Recall that, defining $g' = \tilde{K}'x^*$, and $g = \tilde{K}x^*$, the advantage function
$$A_{K,g}(x,u) = c(x,u) + \gamma V_{K,g}(Ax+Bu) - V_{K,g}(x)$$
satisfies 
\begin{align*}
    &A_{K,g}(x,K'x + g') \\
    &= 2 \tr((K'-K)^\top (C_K x + d_{K,\tilde{K}}x^*)x^\top) + 2\tr((g'-g)^\top (C_K x + d_{K,\tilde{K}}x^*)) \\
    &\quad + 2 \tr((K'-K)^\top (R + \gamma B^\top P_K B)(g'-g) x^\top) \\
    &\quad + \tr(xx^\top(K'-K)^\top (R + \gamma B^\top P_K B)(K'-K)) + (g'-g)^\top (R + \gamma B^\top P_K B)(g'-g) 
\end{align*}
As a result, applying the same analysis from Lemma \ref{lemma:cost-difference} and taking expectations,
\begin{align}
    &J(K',g') - J(K,g) \nonumber \\
    &= \bbE\left[\sum_{t=0}^{\infty} \gamma^t A_{K,g}(x_t, K'x_t + g')\right] \nonumber\\
    &= 2 \tr\left((K'-K)^\top \left(C_K \bbE\left[\sum_{t=0}^{\infty} \gamma^t x_t'x_t' \right]  + d_{K,\tilde{K}}x^* \bbE\left[\sum_{t=0}^{\infty} \gamma^t (x_t')^\top\right]\right)\right) \nonumber\\
    &\quad + 2\tr\left((g'-g)^\top \left(C_K \bbE\left[\sum_{t=0}^{\infty} \gamma^t x_t'\right] + d_{K,\tilde{K}}x^*\right)\right) \nonumber\\
    &\quad + 2 \tr\left((K'-K)^\top (R + \gamma B^\top P_K B)(g'-g) \bbE\left[\sum_{t=0}^{\infty} \gamma^t (x_t')^\top\right]\right) \nonumber\\
    &\quad + \tr\left(\bbE\left[\sum_{t=0}^{\infty} \gamma^t x_t'x_t' \right](K'-K)^\top (R + \gamma B^\top P_K B)(K'-K)\right) + (g'-g)^\top (R + \gamma B^\top P_K B)(g'-g) \nonumber\\
    &= 2 \tr((K'-K)^\top (C_K \Sigma_{K',\tilde{K}'}  + d_{K,\tilde{K}}x^* \rho_{K',\tilde{K}'}^\top)) \nonumber\\
    &\quad + 2\tr((g'-g)^\top (C_K \rho_{K',\tilde{K}'} + d_{K,\tilde{K}}x^*)) \nonumber\\
    &\quad + 2 \tr((K'-K)^\top (R + \gamma B^\top P_K B)(g'-g) \rho_{K',\tilde{K}'}^\top) \nonumber\\
    &\quad + \tr(\Sigma_{K',\tilde{K}'}(K'-K)^\top (R + \gamma B^\top P_K B)(K'-K)) + (g'-g)^\top (R + \gamma B^\top P_K B)(g'-g). \label{eq:value_diff_eq_intermediate_for_smoothness}
\end{align}
We note that $\{x_t'\}$ refers to the trajectory obtained by the following the policy $u = K'x + \tilde{K}'x^*$. To obtain the last equality, we recall the following definitions:
\begin{align*}
    \Sigma_{K',\tilde{K}'} := \bbE\left[\sum_{t=0}^{\infty} \gamma^t x_t'x_t' \right], \quad \rho_{K',\tilde{K}'} := \bbE\left[\sum_{t=0}^{\infty} \gamma^t x_t'\right].
\end{align*}
For notational convenience, we also define
\footnotesize
\begin{align*}
   &\Psi_{(\!K,\tilde{K}\!),(\!K',\tilde{K}'\!)} := \norm*{R \!+\! \gamma B^\top P_K B}_2\!\left(\max\!\left\{\!\norm*{\Sigma_{K',\tilde{K}'}}\!,\! \norm*{\rho_{K',\tilde{K}'}}\!,\! \norm*{x^*}^2 \!\right\}\right).
\end{align*}
\normalsize
Continuing from \eqref{eq:value_diff_eq_intermediate_for_smoothness}, and recalling that $g' = \tilde{K}'x^*, g = \tilde{K}x^*$, we have that
\begin{align*}
    &J(K',\tilde{K}') - J(K,\tilde{K}) = J(K',g') - J(K,g) \\
    &= 2 \tr((K'-K)^\top (C_K \underline{\bm{\Sigma_{K',\tilde{K}'}}}  + d_{K,\tilde{K}}x^* \underline{\bm{\rho_{K',\tilde{K}'}}}^\top)) \\
    & \ \ + 2\tr((\tilde{K}' - \tilde{K})^\top (C_K \underline{\bm{\rho_{K',\tilde{K}'}}} + d_{K,\tilde{K}}x^*)(x^*)^\top) \\
    & \ \ + O\left( \Psi_{(K,\tilde{K}),(K',\tilde{K}')} \left(\norm*{K' - K}^2 + \norm*{\tilde{K}' - \tilde{K}}^2 \right)\right) \\
    &\labelrel\approx{eq:smoothness_approx-appendix} 2 \tr((K'-K)^\top (C_K \underline{\bm{\Sigma_{K,\tilde{K}}}}  + d_{K,\tilde{K}}x^* \underline{\bm{\rho_{K,\tilde{K}}}}^\top)) \\
    & \ \ + 2\tr((\tilde{K}' - \tilde{K})^\top (C_K \underline{\bm{\rho_{K,\tilde{K}}}} + d_{K,\tilde{K}}x^*)(x^*)^\top) \\
    & \ \ + O\left( \Psi_{(K,\tilde{K}),(K',\tilde{K}')} \left(\norm*{K' - K}^2 + \norm*{\tilde{K}' - \tilde{K}}^2 \right)\right) \\
    &\labelrel={eq:gradient_equality-appendix} \tr((K'-K)^\top \nabla_K J(K,\tilde{K}) \\
    & \ \ + \tr((\tilde{K}' - \tilde{K})^\top \nabla_{\tilde{K}} J(K,\tilde{K}) \\
    & \ \ + O\left( \Psi_{(K,\tilde{K}),(K',\tilde{K}')} \left(\norm*{K' - K}^2 + \norm*{\tilde{K}' - \tilde{K}}^2 \right)\right) \\
    &\labelrel={eq:smoothness_simplify_big_O-appendix}\tr((K'-K)^\top \nabla_K J(K,\tilde{K}) \\
    & \ \ + \tr((\tilde{K}' - \tilde{K})^\top \nabla_{\tilde{K}} J(K,\tilde{K}) \\
    & \ \ + O\left( \Psi_C \left(\norm*{K' - K}^2 + \norm*{\tilde{K}' - \tilde{K}}^2 \right)\right)
\end{align*}
\normalsize
Above, we bolded and underlined the terms that are different before and after the approximation in \eqref{eq:smoothness_approx-appendix}. By constraining $\norm*{K' - K}$ and $\norm*{\tilde{K}' - \tilde{K}}$ to both be sufficiently small, it is possible to prove perturbation bounds of the form
\small
\begin{align}
    &\norm*{\Sigma_{K',\tilde{K}'} -  \Sigma_{K,\tilde{K}}} \leq c \left(\norm*{K' - K} + \norm*{\tilde{K}' - \tilde{K}} \right), \label{eq:perturbation_bd_1_appendix}\\ 
    &\norm*{\rho_{K',\tilde{K}'} - \rho_{K,\tilde{K}}} \leq c' \left(\norm*{K' - K} + \norm*{\tilde{K}' - \tilde{K}} \right) \label{eq:perturbation_bd_2_appendix}
\end{align}
\normalsize
for some positive constants $c,c'$ depending only on $(A,B,Q,R,\Sigma, \gamma,x^*,C)$. This shows that the approximation in \eqref{eq:smoothness_approx-appendix} is locally valid. Meanwhile, the equality in \eqref{eq:gradient_equality-appendix} holds by the form of the gradient of $J$ we computed in Lemma \ref{lemma:gradient of J}. Finally, for the equality in \eqref{eq:smoothness_simplify_big_O-appendix}, we replaced the constant $\Psi_{(K,\tilde{K}),(K',\tilde{K}')}$ with a constant $\Psi_C$, which can be shown to depend only on the sublevel set parameter $C$ as well as the system parameters $(A,B,Q,R,\Sigma,\gamma,x^*)$. To see why, recall that
\footnotesize
\begin{align*}
   &\Psi_{(\!K,\tilde{K}\!),(\!K',\tilde{K}'\!)} := \norm*{R \!+\! \gamma B^\top P_K B}_2\!\left(\max\!\left\{\!\norm*{\Sigma_{K',\tilde{K}'}}\!,\! \norm*{\rho_{K',\tilde{K}'}}\!,\! \norm*{x^*}^2 \!\right\}\right).
\end{align*}
\normalsize
Since $(K,\tilde{K})$ lies within the sublevel set $\mathcal{G}_C$, it is possible to uniformly bound the norm of $P_K$, $\Sigma_{K,\tilde{K}}$, as well as $\rho_{K,\tilde{K}}$, over the set $\mathcal{G}_C$. Coupled with perturbation bounds such as the ones in \eqref{eq:perturbation_bd_1_appendix} and \eqref{eq:perturbation_bd_2_appendix}, it is not hard to see that $\Psi_{(K,\tilde{K}),(K',\tilde{K}')}$ can be replaced with a constant $\Psi_C$, which depends only on the sublevel set parameter $C$ as well as the system parameters $(A,B,Q,R,\Sigma,\gamma,x^*)$. From \eqref{eq:smoothness_simplify_big_O-appendix}, we see that for any $(K,\tilde{K})$ in $\mathcal{G}_C$, the cost difference between it and a sufficiently close policy $(K',\tilde{K}')$ can be written as an inner product with the gradient plus some quadratic terms. By deriving norm bounds on $\nabla J(K,\tilde{K})$ uniform over any $(K,\tilde{K})$ over $\mathcal{G}_C$, it follows that $J$ is locally Lipschitz within a radius $\rho$ and Lipschitz parameter $\lambda$ that hold uniformly over $\mathcal{G}_C$. A similar approach allows us to prove the local Lipschitzness of the gradient of $J$. Our argument will thus comprise the following steps:
\begin{enumerate}
    \item Prove a variety of auxiliary norm bounds on terms
    such as $$P_K, C_K, \Sigma_{K,\tilde{K}}, d_{K,\tilde{K}}x^*,\rho_{K,\tilde{K}},$$ within $\mathcal{G}_C$. 
    \item Prove a series of perturbation bounds of the type ``assuming $(K,\tilde{K})$ is in the sublevel set $\mathcal{G}_C $, if $\norm{K'-K}$ and $\norm{\tilde{K}' - \tilde{K}}$ is small enough, then the difference (Term$(K',\tilde{K}')$ - Term $(K,\tilde{K})$) also has small norm'' for various terms Term($K,\tilde{K}$) depending on $K$ and $\tilde{K}$.
    \item Prove the local Lipschitz gradient condition within radius $\rho$, using the preceding norm bounds and perturbation bounds.
    \item Prove local Lipschitzness of $J$ (both population cost and sample cost) using the preceding norm bounds and perturbation bounds.

\end{enumerate}
In the sequel, without loss of generality, we assume that the covariance of the initial noise distribution, $\Sigma$, is the identity matrix $I_{n \times n}$. We also interchangably refer to $J(K',\tilde{K}')$ as $J(K',g')$ and $J(K,\tilde{K})$ as $J(K,g)$, using the notation $g' = \tilde{K}'x^*$ and $g = \tilde{K}x^*$.

\paragraph{Constants used in local Lipschitz and smoothness arguments.} Below is a list of constants that we will use in the local smoothness argument, where $C > 0$ is the value defining the sublevel set $\mathcal{G}_C$.
\small
\begin{itemize}
    \item $c_1 = \frac{\sqrt{(C - C(K^{\mathrm{opt}})) (\rho(R) +\norm{B}^2C)}}{\sigma_{\min}(R)} + \frac{\norm{B}C \norm{A}}{\sigma_{\min}(R)}$ \\
    \item $c_2 = \sqrt{\frac{C}{\sigma_{\min}(R)}}$ 
    \item $c_2' = c_2 + \norm{x^*}$
    \item $c_3 = \frac{2C}{\sigma_{\min}(Q)}$
    \item $c_4 = \frac{\sigma_{\min}(Q)}{4C \norm{\sqrt{\gamma}B}_2(\norm{\sqrt{\gamma}A}_2 + \norm*{\sqrt{\gamma}B}_2c_1 + 1)}$
    \item $c_5 = \frac{\sigma_{\min}(Q)}{4\gamma \norm{B} C}$.
    \item $c_6 = \frac{1}{\sigma_{\min}(R)}\left(\sqrt{(\norm{R}_2 + \norm{B}_2^2 C) (C - J^{\mathrm{opt}})} + \norm{B}_2 \norm{A}_2 C \right)$
    \item $c_7 = \min \left\{\frac{C}{\sigma_{\min}(Q)} \sqrt{(\norm{R}_2 + \norm{B}_2^2 C) (C - J^{\mathrm{opt}})}, c_6 \right\}$ 
    \item $c_8 = 4 \left( \frac{C}{\sigma_{\min}(Q)}\right)^2 \norm{Q}_2 \norm{B}_2 (\norm{A}_2 + \norm{B}_2c_{7} + 1)$ 
    \item $c_9 = 8 \left( \frac{C}{\sigma_{\min}(Q)}\right)^2 (c_7^2) \norm{R}_2\norm{B}_2 (\norm{A}_2 + \norm{B}_2c_{7} + 1)$
    \item $c_{10} = 2 \left( \frac{C}{\sigma_{\min}(Q)}\right)^2 (c_7 + 1) \norm{R}_2$
    \item $c_{11} = c_8 + c_9 + c_{10}$
    \item $c_{12} = 2\gamma c_2' c_3 \norm{B}_2 (c_1 + 1)\norm{R}_2 + c_2' \norm{R}_2$ 
    \scriptsize
    \item 
    $c_{13} = 2 \gamma^{3/2} c_2' (C + c_{11}) \norm{B}_2^2  
        \!+\! c_2' c_3(\gamma \norm{B}_2^2 (C + c_{11}) \norm{B}_2 + \sqrt{\gamma} \norm{B}_2 c_2'c_{11})$
    \small
    \item $c_{14} = 2\gamma c_3 \norm{B}_2 \norm{Qx^*} + c_{12} + c_{13}$
    \item $c_{15} = 2c_1 c_3 \norm{R}_2 + \gamma c_3 C \norm{B}_2$.
    \item $c_{16} = \norm{R}_2 + \gamma \norm{B}_2 [c_{11}+ \norm{B}_2 C)] $
    \item $c_{17} = \gamma \norm{B}_2^2 c_2' c_{11} + \gamma \norm{B}_2 c_{14}$ 
    \item $c_{18} = \norm{R}_2 + \gamma \norm{B}_2^2 C + \gamma \norm{B}_2 c_{15}$
\end{itemize}
\normalsize

\subsubsection{Auxiliary bounds}

We first bound the terms $K$ and $g := \tilde{K}x^*$. It is helpful to build on existing results for the LQR cost with zero target in \cite{DBLP:conf/icml/Fazel0KM18}.
We define $f(K)$ to be the following cost, a standard LQR with tracking target $x^* = 0$: 
\begin{equation}
    \begin{array}{ll}
         & f(K) = \sum_{t=0}^{\infty} x_t^\top Q x_t + u_t^\top R u_t,  \\
        \mbox{such that} & x_{t+1} = \sqrt{\gamma}Ax_t + \sqrt{\gamma}Bu_t, \\
        & u_t = Kx_t, \quad x_0 \sim N(0,\Sigma).
    \end{array}
    \label{equation: f(K) definition}
\end{equation}
We now show that if $J(K,\tilde{K})$ is bounded, then so is the cost $f(K)$.
\begin{lemma}
\label{lemma:bound_on_J_implies_bound_on_f}
Consider any constant $C > 0$. Suppose $J(K,\tilde{K}) \leq C$. Then, $f(K) \leq C$ holds.
\end{lemma}
\begin{proof}
Recall the notation $g := \tilde{K}x^*$. Based on Lemma \ref{lemma:value function of J(K,g)}, we see that 
\begin{align*}
    J(K,g) &= \tr(\Sigma P_K) + \mathbb{E}[x_0]^\top q_{K,g} +  r_{K,g}, \\
    &= \tr(\Sigma P_K) + r_{K,g} = f(K) + r_{K,g},
\end{align*}
where we used $\mathbb{E}[x_0] = 0$ as well as the easy-to-check fact that $f(K) = \tr(\Sigma P_K)$. Since $r_{K,g} \geq 0$ has to hold --- otherwise the value function $V(K,g; 0)$, \emph{i.e.} value function of $(K,g)$ when $x_0 = 0$, is negative ---,  it follows that $$J(K,\tilde{K}) = J(K,g) = f(K) + r_{K,g} \leq C \implies f(K) \leq C.$$
\end{proof}

Consider the zero-target trajectory $\{y_t\}$ given by the dynamics $y_0 =x_0, y_{t+1} = \sqrt{\gamma}Ay_t + \sqrt{\gamma}BKy_t$. Define 
$$\Phi_K := \bbE\left[\sum_{t=0}^{\infty} y_t y_t^\top\right].$$

With these definitions out of the way, we can turn to our first bounds on $\norm{K}$ and $\norm{\tilde{K}x^*}.$
\begin{lemma}[Bound on $\norm{K}, \norm{\tilde{K}x^*}$]
\label{lemma:bound K and g}
Suppose $(K, \tilde{K}) \in \mathcal{G}_C$. Then, there exists constants $c_1, c_2 > 0$ depending only on $(A,B,Q,R,\Sigma,\gamma,C)$ such that 
\begin{align*}
    \norm{K} \leq c_1, \norm{\tilde{K}x^*} \leq c_2. 
\end{align*}
\end{lemma}
\begin{proof}
Since $J(K,\tilde{K}) \leq C$, $f(K) \leq C$ holds as well by Lemma \ref{lemma:bound_on_J_implies_bound_on_f}. Thus, by Lemma 25 in \cite{DBLP:conf/icml/Fazel0KM18}, there exists $c_1$ depending only on $(A,B,Q,R,\Sigma,\gamma,C)$ such that $\norm{K} \leq c_1$. 

We next show that $g = \tilde{K}x^*$ is bounded. To this end, note that the cost at $t = 0$ is bounded by $C$, namely
\begin{align*}
    &\mathbb{E}[(x_0 - x^*)^\top Q (x_0 - x^*) + (Kx_0 + g)^\top R(Kx_0 + g)] \leq C  \\
    &\implies \tr(Q) + (x^*)^\top Qx^* + x_0^\top (K^\top RK)x_0 + g^\top Rg \leq C \\
    &\implies g^\top R g \leq C \\
    &\implies \norm{g} \leq \sqrt{C/\sigma_{\min}(R)} := c_2.
\end{align*}
\end{proof}

We next bound $\tr(\Phi_K)$ and $\norm{P_K}_2.$ This result follows directly from Lemma 13 in \cite{DBLP:conf/icml/Fazel0KM18}. \begin{lemma}[Bound on $\tr(\Phi_K)$ and $\norm{P_K}_2$]
\label{Lemma:bound Phi_k and P_K}
Suppose $(K, \tilde{K}) \in \mathcal{G}_C$. Then, 
\begin{align*}
    \norm{P_K}_2 \leq C, \quad  \tr(\Phi_K) \leq \frac{C}{\sigma_{\min}(Q)}.
\end{align*}
\end{lemma}

We next turn to bounding the following terms: $$C_K, \Sigma_{K,\tilde{K}}, d_{K,\tilde{K}}x^*, \rho_{K,\tilde{K}}.$$

Recall that
\begin{align*}
  C_K = RK + \gamma B^\top P_K (BK + A), \quad d_{K,\tilde{K}} x^*= (R\tilde{K} + \gamma B^\top P_K B\tilde{K} + \gamma B^\top q_{K,\tilde{K}})x^*.
\end{align*}
\begin{lemma}[Bound on $C_K$ and $d_{K,\tilde{K}}x^*$]
\label{lemma:bound_C_K_d_KKtilde}
Suppose $$(K, \tilde{K}) \in \mathcal{G}_C = \{(K, \tilde{K}) \mid J(K,\tilde{K}) \leq C\}.$$ Then, 
\begin{align*}
    \norm*{C_K}_F^2 \leq C \norm{R + B^\top P_K B}_2, \quad \norm{d_{K,\tilde{K}}x^*}_F^2 \leq C \norm{R + B^\top P_K B}_2.
\end{align*}
\end{lemma}
\begin{proof}
From the lower bound in Proposition \ref{proposition: gradient dominant single agent}, we see that 
\begin{align*}
    \norm*{C_K}_F^2 + \norm{d_{K,\tilde{K}}x^*}_F^2 \leq (J(K,\tilde{K}) - J(K^{\mathrm{opt}}, \tilde{K}^{\mathrm{opt}})) \norm{R + B^\top P_K B}_2 \leq J(K,\tilde{K}) \norm{R + B^\top P_K B}_2.
\end{align*}
Using the definition of the sublevel set $\mathcal{G}_C$ then yields the result. 
\end{proof}

We next seek to bound $\Sigma_{K,\tilde{K}}$ and $\rho_{K,\tilde{K}}$. We begin with $\rho_{K,\tilde{K}}$. Observe that
\begin{align*}
    \rho_{K,\tilde{K}} = \bbE\left[\sum_{t=0}^{\infty} \gamma^t x_t \right],
\end{align*}
where $\{x_t\}$ is generated by the policy $(K, \tilde{K})$. Using the fact that $\bbE[x_0] = 0$, since $x_t = (A+BK)^t x_0 + \sum_{i=0}^{t-1} (A+BK)^i B\tilde{K}x^*$, 
we see that 
\begin{align*}
    \bbE[x_t] = \sum_{i=0}^{t-1} (A+BK)^i B\tilde{K} x^*.
\end{align*}
With some algebraic manipulation, introducing the definition
$$S_K := (I - \gamma (A+BK))^{-1},$$
we see that 
\begin{align}
    \rho_{K,\tilde{K}} &= \sum_{t=0}^{\infty} \gamma^t \left( \sum_{i=0}^{t-1} (A+BK)^i B\tilde{K} x^*\right) \nonumber \\
    &\labelrel={eq:rho_K_Ktilde_manipulate} \sum_{t=1}^{\infty}\gamma^t \left(\sum_{i=0}^{\infty} (A+BK)^i \right) B\tilde{K}x^* \nonumber \\
    &= \sum_{t=1}^{\infty} \gamma^t S_K B\tilde{K}x^*. \label{equation:rho_KL_halfway_bound} 
\end{align}
Above, \eqref{eq:rho_K_Ktilde_manipulate} follows from rearranging the terms in the summation. Thus bounding $\rho_{K,\tilde{K}}$ reduces to bounding $S_K$ and $B\tilde{K}x^*$. We already have a bound on the latter term from Lemma \ref{lemma:bound K and g}. This next result bounds $S_K $.
\begin{lemma}[Bound on $S_K$] \label{lemma: bounding S_K}
Suppose $(K, \tilde{K}) \in \mathcal{G}_C$. Let 
$$S_K := (I - \gamma (A+BK))^{-1}.$$ Then, 
\begin{align*}
    \norm{S_K}_2 \leq  \frac{2C}{\sigma_{\min}(Q)} := c_3.
\end{align*}
\end{lemma}
\begin{proof}
Observe now that
\begin{align*}
    \norm{S_K}_2 &= \rho\left(( I - \gamma(A+BK))^{-1}\right) \\
    &= \rho\left(\sum_{t=0}^{\infty} \gamma^t(A+BK)^t\right) \\
    &\labelrel\leq{eq:bound_S_k_part_1} 2 \sum_{t=0}^{\infty} \gamma^{2t} \rho\left((A+BK)^{2t}\right) \\
    &\leq 2 \sum_{t=0}^{\infty} \gamma^{t} \rho\left((A+BK)^{2t}\right) \\
    &\leq 2 \sum_{t=0}^{\infty}  \tr\left( (\sqrt{\gamma}(A+BK))^t((\sqrt{\gamma}(A+BK))^t)^\top\right) \\
    &\labelrel={eq:bound_S_k_part_2} 2 \sum_{t=0}^{\infty} \gamma^{t} \tr\left( (\sqrt{\gamma}(A+BK))^t\Sigma((\sqrt{\gamma}(A+BK))^t)^\top\right) \\
    &\labelrel\leq{eq:bound_S_k_part_3} 2 \tr(\Phi_K) \labelrel\leq{eq:bound_S_k_part_4} 2 \frac{f(K)}{\sigma_{\min}(Q)} \labelrel\leq{eq:bound_S_k_part_5} 2 \frac{C}{\sigma_{\min}(Q)},
\end{align*}
Above, inequality \eqref{eq:bound_S_k_part_1} uses the fact that $\sqrt{\gamma}\rho(A+BK) < 1$ so that the even terms in the sum are larger than the odd terms, \eqref{eq:bound_S_k_part_2} uses our simplifying assumption made in the outline that $\Sigma := \sum_{t=0}^{\infty} x_0 x_0^\top = I$, \eqref{eq:bound_S_k_part_3} follows from the definition of $\Phi_K:= \sum_{t=0}^{\infty} y_t y_t^\top$, where $\{y_t\}$ is the zero-target trajectory given by the dynamics $y_0 =x_0, y_{t+1} = \sqrt{\gamma}Ay_t + \sqrt{\gamma}BKy_t$, such that $y_t = (\sqrt{\gamma}(A+BK))^t x_0$. Meanwhile, inequality \eqref{eq:bound_S_k_part_4} comes from Lemma \ref{Lemma:bound Phi_k and P_K}, and inequality \eqref{eq:bound_S_k_part_5} follows from Lemma \ref{lemma:bound_on_J_implies_bound_on_f}.
\end{proof}

We are now ready to bound $\rho_{K,\tilde{K}}x^*$. 
\begin{lemma}[Bound on $\rho_{K,\tilde{K}}x^*$]
\label{lemma:bound on rho_KL}
Suppose $(K, \tilde{K}) \in \mathcal{G}_C$. Then, recalling the definition
$$\beta_{\gamma} = \sum_{t=0}^{\infty} \gamma^t, $$
we have
\begin{align*}
    \norm{\rho_{K,\tilde{K}}x^*} \leq 2 \beta_{\gamma} \frac{C}{\sigma_{\min}(Q)} \norm{B}c_2.
\end{align*}
\end{lemma}
\begin{proof}
From Equation \ref{equation:rho_KL_halfway_bound}, we have the following bound
\begin{align*}
\norm{\rho_{K,\tilde{K}}x^*} &= \norm*{\sum_{t=1}^{\infty} \gamma^t S_K B\tilde{K}x^*} \\
&\leq \sum_{t=1}^{\infty} \gamma^t \norm{S_K B\tilde{K}x^*} \\
&\leq \sum_{t=1}^{\infty} \gamma^t \norm{S_K}_2 \norm{B} \norm{\tilde{K}x^*} \\
&\leq \sum_{t=1}^{\infty} 2 \frac{f(K)}{\sigma_{\min}(Q)} \norm{B}c_2 \leq 2 \beta_{\gamma} \frac{C}{\sigma_{\min}(Q)} \norm{B}c_2
\end{align*}
The penultimate inequality comes from the bounds on $\norm{S_K}_2$ and $\norm{\tilde{K}x^*}$ from Lemma \ref{lemma: bounding S_K} and Lemma \ref{lemma:bound K and g} respectively, while the final inequality comes from our definition of $\beta_{\gamma}$ and $f(K) \leq C$ which holds by Lemma \ref{lemma:bound_on_J_implies_bound_on_f}.
\end{proof}

We can now turn to bounding $\Sigma_{K,\tilde{K}}$. Consider defining the following linear operator on symmetric matrices, $T_K(\cdot)$, as follows
\begin{align*}
    T_K(X) := \sum_{t=0}^{\infty} \gamma^t (A+BK)^t X [(A+BK)^\top]^t,
\end{align*}
and define the induced norm of $T$ as follows:
\begin{align*}
    \norm{T_K} = \sup_X \frac{\norm{T_K(X)}}{\norm{X}}.
\end{align*}
We may bound the norm of $T_K$ as follows. We omit the proof since it follows directly from Lemma 17 in \cite{DBLP:conf/icml/Fazel0KM18} applied to the system matrices $(\sqrt{\gamma}A, \sqrt{\gamma}B)$, as well as our bound on $f(K)$ from Lemma \ref{lemma:bound_on_J_implies_bound_on_f}.
\begin{lemma}[Bound on $T_K$ norm]
\label{lemma:bound on T_K norm}
We have
\begin{align*}
    \norm{T_K} \leq \frac{f(K)}{ \sigma_{\min}(Q)} \leq \frac{C}{ \sigma_{\min}(Q)} .
\end{align*}
\end{lemma}

\begin{lemma}[Bound on $\Sigma_{K,\tilde{K}}$]
\label{lemma:bound on Sigma_K,Ktilde}
Suppose $(K, \tilde{K}) \in \mathcal{G}_C$. Then,
\begin{align*}
    \Sigma_{K,\tilde{K}} = \mathbb{E} \left[\sum_{t=0}^{\infty} \gamma^t x_t x_t^\top \right] = \Phi_K + \sum_{t=1}^{\infty} \gamma^t T_K \left(Bg (Bg)^\top  S_K^\top\right).
\end{align*}
Moreover, 
\begin{align*}
    \norm{\Sigma_{K,\tilde{K}}} \leq \frac{C}{\sigma_{\min}(Q)} + \beta_{\gamma} \frac{C}{\sigma_{\min}(Q)}\norm{B}^2 c_2^2.
\end{align*}
\end{lemma}
\begin{proof}
    Again, recall that $g = \tilde{K}x^*$.
    Observe that
    \begin{align*}
        \mathbb{E}[x_t x_t^\top] = \mathbb{E}[(x_t - \mathbb{E}[x_t])(x_t - \mathbb{E}[x_t])^\top] + \mathbb{E}[x_t]\mathbb{E}[x_t]^\top.
    \end{align*}
    We first find that
    \begin{align*}
        \mathbb{E}\left[\sum_{t=0}^{\infty}\gamma^t(x_t - \mathbb{E}[x_t])(x_t - \mathbb{E}[x_t])^\top \right] &= \mathbb{E}\left[\sum_{t=0}^{\infty} y_ty_t^\top\right] = \Phi_K,
    \end{align*}
    where the sequence $\{y_t\}$ is the zero-target trajectory obtained by the dynamics $y_{t+1} = \sqrt{\gamma} Ay_t + \sqrt{\gamma} BK y_t$,
    and we used the fact that each summand in the terms involving $g$ cancel out. Next, we have 
    \begin{align*}
       \sum_{t=0}^{\infty} \gamma^t \bbE[x_{t}]\bbE[x_t^\top] &= \sum_{t=0}^{\infty} \gamma^t \left(\sum_{i=0}^{t-1} (A+BK)^iBg\right)\left(\sum_{i=0}^{t-1} (A+BK)^iBg\right)^\top \\
       &= \sum_{t=1}^{\infty} \gamma^t Bg (Bg)^\top \left(\sum_{i=0}^{t-1} (A+BK)^i \right)^\top\\
       &\quad + \gamma(A+BK) \left(\sum_{t=1}^{\infty} \gamma^{t} Bg (Bg)^\top \left(\sum_{i=0}^{t-1} (A+BK)^i \right)^\top\right) (A+BK)^\top \\
       &\quad + \gamma^2 (A+BK)^2 \left(\sum_{t=1}^{\infty} \gamma^{t} Bg (Bg)^\top \left(\sum_{i=0}^{t-1} (A+BK)^i \right)^\top\right) ((A+BK)^2)^\top \\
       &\quad + \cdots \\
       &= T_K \left(Bg (Bg)^\top \sum_{t=1}^{\infty} \left(\sum_{i=0}^{t-1} (A+BK)^i \right)^\top\right) \\
       &= T_K \left(Bg (Bg)^\top \sum_{t=1}^{\infty} \gamma^t S_K^\top\right)
    \end{align*}
Thus, combining, we find
\begin{align*}
    \norm{\Sigma_{K,\tilde{K}}} &\leq \norm{\Phi_K} + \norm*{T_K \left(Bg (Bg)^\top \sum_{t=1}^{\infty} \gamma^t S_K^\top\right)} \\
    &\leq \norm{\Phi_K} + \sum_{t=1}^{\infty} \gamma^t \norm*{T_K \left(Bg (Bg)^\top  S_K^\top\right)} \\
    &\leq \norm{\Phi_K} + \beta_{\gamma} \norm{T_K} \norm{Bg}^2 \norm*{S_K^\top}_2^2 \\
    &\leq \frac{C}{\sigma_{\min}(Q)} + \beta_{\gamma} \frac{C}{ \sigma_{\min}(Q)}\norm{B}^2 c_2^2 c_3^2.
\end{align*}
For the final inequality, we used Lemma \ref{lemma:bound K and g}, Lemma \ref{Lemma:bound Phi_k and P_K}, Lemma \ref{lemma: bounding S_K} and  Lemma \ref{lemma:bound on T_K norm} to bound the norms of $g$, $\Phi_K$, $S_K$, and $T_K$ respectively. 
\end{proof}

\subsubsection{Helpful perturbation bounds}

We now proceed to state and prove various perturbation bounds of the type ``assuming $(K,\tilde{K})$ is in the sublevel set $\mathcal{G}_C$, if $\norm{K'-K}$ and $\norm{\tilde{K}' - \tilde{K}}$ is small enough, then Term($(K',\tilde{K}')$ - Term $(K,\tilde{K})$ also has small norm.'' for various terms Term($K,\tilde{K}$) depending on $K$ and $\tilde{K}$.

Throughout, we first need to establish that the $K'$ we pick has finite cost. To this end, consider the following lemma.

\begin{lemma}[Local radius for stable $K'$]
\label{lemma:rho(A+BK') < 1 }
    Suppose 
    \begin{align*}
        \norm{K' - K} \leq \frac{\sigma_{\min}(Q)}{4C \norm{\sqrt{\gamma}B}_2(\norm{\sqrt{\gamma}A}_2 + \norm*{\sqrt{\gamma}B}_2c_1 + 1)} := c_4.
    \end{align*}
Then, $\sqrt{\gamma} (A+BK') < 1$.
\end{lemma}
\begin{proof}
    This follows from the same argument in Lemma 22 of \cite{DBLP:conf/icml/Fazel0KM18}, with the system $(A,B)$ replaced by $(\sqrt{\gamma}A, \sqrt{\gamma}B)$. 
\end{proof}

Next, we seek to bound $(S_{K'} - S_K)$ --- to be precise, we show when $K - K'$ is small enough, $\norm{S_K' - S_K}$ can be bounded in terms of $\norm{K'-K}$.

\begin{lemma}
\label{lemma:bound S_K' - S_K}
Suppose $(K,\tilde{K})$ is in $\mathcal{G}_C$. Suppose $\norm{K'-K} \leq c_4$, so $\sqrt{\gamma} \rho(A + BK') < 1$, and that $\norm{K'-K} \leq \frac{\sigma_{\min}(Q)}{4\gamma \norm{B} C} := c_5$.
 Then, 
\begin{align*}
    \norm{S_{K'} - S_K} \leq 2\gamma c_3 \norm{B}_2 \norm{K' - K}, \quad \mbox{and} \quad \norm{S_{K'}}_2 \leq 2\norm{S_K}_2.
\end{align*}
\end{lemma}
\begin{proof}
For notational convenience, in this proof, we define
\[D := I - \gamma(A+BK), \quad E := \gamma B(K-K').\]
Then, we have
\begin{align}
    S_{K'} - S_K &= (D + E)^{-1} - D^{-1} \nonumber \\
    &= D^{-1}\left( I + ED^{-1}\right)^{-1} - D^{-1} \nonumber \\
    &= D^{-1} \left((I + ED^{-1})^{-1} - I\right) \label{eq:S_K'-S_K-bdd-first-eqn}
\end{align}
Next, consider $I + ED^{-1}$. Observe that
\begin{align}
    (I + ED^{-1})^{-1} = I - (I + ED^{-1})^{-1}ED^{-1}, \label{eq:bounding_I_EDinv_inv}
\end{align}
which can be verified by multiplying both sides with $I + ED^{-1}$. Note now that
\begin{align}
    \norm{
    ED^{-1}} &= \norm*{\gamma B(K-K') S_K} \labelrel\leq{eq:bound_on_EDinv_K'-K} \norm{S_K}_2 (\gamma \norm{B}_2 \norm{K' - K}) \labelrel\leq{eq:bound_on_EDinv_last_part} \frac{1}{2}, \label{eq:bound_on_EDinv}
\end{align}
where the inequality in \eqref{eq:bound_on_EDinv_last_part} holds because $\norm{S_K}_2 \leq \frac{2C}{\sigma_{\min}(Q)}$ from Lemma \ref{lemma: bounding S_K} and our assumption that $\norm{K'-K} \leq \frac{\sigma_{\min}(Q)}{4\gamma \norm{B} C}$. Continuing with equation \eqref{eq:bounding_I_EDinv_inv}, we note that
\begin{align}
    &\quad (I + ED^{-1})^{-1} = I - (I + ED^{-1})^{-1}ED^{-1} \nonumber \\
    &\implies \norm*{(I + ED^{-1})^{-1}}_2 \leq \norm*{I}_2 + \norm*{(I + ED^{-1})^{-1}ED^{-1}}_2 \nonumber \\
    &\labelrel\implies{eq:using_bound_EDinv} \norm*{(I + ED^{-1})^{-1}}_2 \leq \norm*{I}_2 + \frac{1}{2}\norm*{(I + ED^{-1})^{-1}}_2 \nonumber \\
    &\implies \frac{1}{2}\norm*{(I + ED^{-1})^{-1}}_2 \leq 1 \implies \norm*{(I + ED^{-1})^{-1}}_2 \leq 2. \label{eq:bound_on_I_EDinv_inv}
\end{align}
Above, \eqref{eq:using_bound_EDinv} follows from \eqref{eq:bound_on_EDinv}.

 Therefore, combining equations \eqref{eq:S_K'-S_K-bdd-first-eqn} and \eqref{eq:bounding_I_EDinv_inv}, we have that
\begin{align*}
    S_{K'} - S_K &= D^{-1} \left((I + ED^{-1})^{-1} - I\right) =  D^{-1} (-(I + ED^{-1})^{-1}ED^{-1}),
\end{align*}
which implies that
\begin{align*}
    \norm{S_{K'} - S_K} &= \norm{D^{-1} (-(I + ED^{-1})^{-1}ED^{-1})} \\
    &\labelrel\leq{eq:S_K'-S_K-other-bdd} \norm{S_K}_2 \norm*{(I + ED^{-1})^{-1}}_2 \norm*{ED^{-1}} \\
    &\leq 2\gamma c_3  \norm{B}_2 \norm{K' - K}
\end{align*}
Above, the final inequality follows from the bound on $\norm{S_K}_2$ from Lemma \ref{lemma: bounding S_K}, the bound on $\norm*{(I + ED^{-1})^{-1}}_2$ from \eqref{eq:bound_on_I_EDinv_inv}, and the bound on $\norm*{ED^{-1}}$ in \eqref{eq:bound_on_EDinv_K'-K} of equation \eqref{eq:bound_on_EDinv}. Meanwhile, it follows from \eqref{eq:S_K'-S_K-other-bdd} and the bounds in \eqref{eq:bound_on_EDinv} and \eqref{eq:bound_on_I_EDinv_inv} that 
\begin{align*}
    \norm*{S_{K'} - S_K}_2 \leq \norm*{S_K}_2.
\end{align*}
\end{proof}

Next, we bound the norm of $(T_{K'} - T_K)$ when $K'$ is close to $K$. The result follows from Equation (30b) of \cite{malik2019derivative}, with the system matrices $(A,B)$ replaced by $(\sqrt{\gamma}A, \sqrt{\gamma}B)$.
\begin{lemma}[Bound on norm of $(T_{K'} - T_{K})$]
\label{Lemma:bound T_K - T_K'}
Suppose $(K, \tilde{K}) \in \mathcal{G}_C$, and that $\norm{K'-K} \leq c_4$. Then,
\begin{align*}
    \norm{T_{K'} - T_K} \leq c_5 \norm{K'-K},
\end{align*}
where $c_5 := \frac{\sigma_{\min}(Q)}{4\gamma \norm{B}C}$ was defined earlier. 
\end{lemma}

Next, we bound the norm of $(P_K - P_{K'})$ and $f(K') - f(K)$ when $K'$ is close to $K$. The result follows from Lemma 13 of \cite{DBLP:conf/icml/Fazel0KM18}, with the system matrices $(A,B)$ replaced by $(\sqrt{\gamma}A, \sqrt{\gamma}B)$.
\begin{lemma}[Bound on norm of $P_{K'} - P_K$ and $f(K') - f(K)$]
\label{Lemma:bound P_K' - P_K and f(K') - f(K)}
Suppose $(K, \tilde{K}) \in \mathcal{G}_C$, and that $\norm{K'-K} \leq \min \{c_4,c_5\}$. Then, there exists a constant $c_{11} > 0$ depending only on $(A,B,Q,R,\Sigma,x^*,\gamma,C)$ such that
\begin{align*}
    \norm{P_{K'} - P_K}_2 \leq c_{11} \norm{K'-K},\quad 
    \abs{f(K') - f(K)} \leq n c_{11} \norm{K'-K}.
\end{align*}
\end{lemma}

We turn now to bounding $\rho_{K', \tilde{K}'} - \rho_{K,\tilde{K}}$. \begin{lemma}[Bound on $\rho_{K', \tilde{K}'} - \rho_{K,\tilde{K}}$]
\label{lemma: bound rho_K'Ktilde' - rho_KKtilde}
Suppose $$(K, \tilde{K}) \in \mathcal{G}_C, \quad \norm{K'-K} \leq \min\{c_4,c_5\}.$$ Then,
\begin{align*}
    \norm{\rho_{K', \tilde{K}'} - \rho_{K,\tilde{K}}} &\leq 4 \gamma \beta_{\gamma} \frac{C}{\sigma_{\min}(Q)}\norm{B}^2 c_2 \norm{K'-K}  + \beta_{\gamma} \frac{4C}{\sigma_{\min}(Q)} \norm{B} \norm{x^*} \norm{\tilde{K}' - \tilde{K}}.
\end{align*}
\end{lemma}
\begin{proof}
Note that
\begin{align*}
    \rho_{K', \tilde{K}'} - \rho_{K,\tilde{K}} &= \sum_{t=1}^{\infty} \gamma^t S_{K'} B\tilde{K}'x^* - \sum_{t=1}^{\infty} \gamma^t S_{K} B\tilde{K}x^* \\
    &= \sum_{t=1}^{\infty} \gamma^t \left(S_{K'}B\tilde{K} x^* + S_{K'} B(\tilde{K}' - \tilde{K})x^*\right) - \sum_{t=1}^{\infty} \gamma^t S_{K} B\tilde{K}x^* \\
    &= \sum_{t=1}^{\infty} \gamma^t (S_{K'} - S_K) B\tilde{K}x^* + \sum_{t=1}^{\infty} \gamma^t S_{K'} B(\tilde{K}' - \tilde{K})x^*.
\end{align*}
Then,
\begin{align*}
    \norm*{\rho_{K', \tilde{K}'} - \rho_{K,\tilde{K}}} &= \norm*{\sum_{t=1}^{\infty} \gamma^t (S_{K'} - S_K) B\tilde{K}x^* + \sum_{t=1}^{\infty} \gamma^t S_{K'} B(\tilde{K}' - \tilde{K})x^*} \\
    &\leq \norm*{\sum_{t=1}^{\infty} \gamma^t (S_{K'} - S_K) B\tilde{K}x^*} + \norm*{\sum_{t=1}^{\infty} \gamma^t S_{K'} B(\tilde{K}' - \tilde{K})x^*} \\
    &\leq \beta_{\gamma} 2\gamma \norm{S_K}_2 \norm{B} \norm{K'-K} \norm{B} \norm{\tilde{K}x^*} + \beta_{\gamma} \norm{S_{K'}} \norm{B}\norm{x^*} \norm{\tilde{K}' - \tilde{K}} \\
    &\leq 4 \gamma \beta_{\gamma} \frac{C}{\sigma_{\min}(Q)}\norm{B}^2 c_2 \norm{K'-K}  + \beta_{\gamma} \frac{4C}{\sigma_{\min}(Q)} \norm{B} \norm{x^*} \norm{\tilde{K}' - \tilde{K}}.
\end{align*} 
The penultimate inequality comes from the perturbation bound on $\norm*{S_{K'} - S_K}_2$ in Lemma \ref{lemma:bound S_K' - S_K}, our assumption that $\norm{K'-K} \leq c_5$, and the norm bound on $\tilde{K}x^*$ in Lemma \ref{lemma:bound K and g}. For the final inequality, we used the fact that $\norm{S_{K'}}_2 \leq 2 \norm{S_K}_2$, which we derived in Lemma \ref{lemma:bound S_K' - S_K}. 
\end{proof}

We next bound $\Sigma_{K', \tilde{K}'} - \Sigma_{K,\tilde{K}}$. 
\begin{lemma}[Bound on $\Sigma_{K', \tilde{K}'} - \Sigma_{K,\tilde{K}}$]
\label{lemma: bound Sigma_K'Ktilde' - Sigma_KKtilde}
Suppose $$(K, \tilde{K}) \in \mathcal{G}_C, \quad \norm{K'-K} \leq \min\{c_4,c_5,1\},\quad \norm*{\tilde{K}' - \tilde{K}} \leq 1.$$ Then, 
\small
\begin{align*}
    &\norm{\Sigma_{K', \tilde{K}'} - \Sigma_{K,\tilde{K}}}\\
    &\leq \left(c_6 + \norm{B}^2 (c_2')^2)(2\gamma c_3 \norm{B}_2) + (4c_3^2) c_5 \norm{B}^2 (c_2')^2  \right) \norm{K'-K} + \beta_{\gamma} \left(\frac{c_3C\norm{x^*}}{\sigma_{\min}(Q)} \left(\norm*{B}^2 c_2'+ c_2\right)\right) \norm*{\tilde{K}'- \tilde{K}}.
\end{align*}
\normalsize
\end{lemma}
\begin{proof}
Note that
\small
\begin{align*}
    \Sigma_{K', \tilde{K}'} - \Sigma_{K,\tilde{K}} &= \left(\Phi_{K'} + \sum_{t=1}^{\infty} \gamma^t T_{K'}\left((B\tilde{K}'x^*)(B\tilde{K}'x^*)^\top S_{K'}^\top\right)\right) - \left(\Phi_{K} + \sum_{t=1}^{\infty}\gamma^t T_K\left((B\tilde{K}x^*)(B\tilde{K}x^*)^\top S_K^\top \right)\right).
\end{align*}
\normalsize
We first bound $\Phi_{K'} - \Phi_K$. To this end, observe that by Lemma 16 in \cite{DBLP:conf/icml/Fazel0KM18}, there exists $c_6 > 0$ depending only on $(A,B,Q,R,\Sigma,\gamma,C)$ such that
\begin{align}
   \norm*{\Phi_{K'} - \Phi_K} \leq c_6 \norm{K'-K}. \label{equation: bound on Phi_k' - Phi_K} 
\end{align}
Meanwhile, recalling that $g := \tilde{K}x^*$ and $g' := \tilde{K}' x^*$, observe that by our assumption on $\norm*{\tilde{K}' - \tilde{K}} \leq 1$, it follows that
\begin{align*}
    \norm*{g' - g} \leq \norm*{\tilde{K}' - \tilde{K}} \norm{x^*} \leq \norm{x^*},
\end{align*}
which implies that
\begin{align}
    \norm*{g'} \leq \norm{g} + \norm{x^*} \leq c_2 + \norm{x^*} := c_2', \label{eq:bound_on_g'}
\end{align}
where the bound on $\norm{g}$ comes from Lemma \ref{lemma:bound K and g}. Then,
\footnotesize
\begin{align*}
    &\norm*{\sum_{t=1}^{\infty} \gamma^t \left(T_{K'}\left((B\tilde{K}'x^*)(B\tilde{K}'x^*)^\top S_{K'}^\top \right) -  T_K\left((B\tilde{K}x^*)(B\tilde{K}x^*)^\top S_K^\top \right)\right)} \\
    &\leq \sum_{t=1}^{\infty} \gamma^t  \norm*{T_{K'}\left((B\tilde{K}'x^*)(B\tilde{K}'x^*)^\top S_{K'}^\top \right) -  T_K\left((B\tilde{K}x^*)(B\tilde{K}x^*)^\top S_K^\top \right)} \\
    &\leq \beta_{\gamma} \left( \norm*{T_{K}\left((Bg')(Bg')^\top S_{K'}^\top\right) - (T_{K} - T_{K'})\left((Bg')(Bg')^\top S_{K'}^\top \right) - T_K\left((Bg)(Bg)^\top S_K^\top \right)} \right) \\
    &\leq \beta_{\gamma} \left(\norm{T_K} \norm*{(Bg')(Bg')^\top S_{K'}^\top - (Bg)(Bg)^\top S_K^\top)} + \norm*{(T_{K} - T_{K'})((Bg')(Bg')^\top)}\right) \\
    &\leq \beta_{\gamma} \left(\norm{T_K} \norm*{(Bg')(Bg')^\top S_{K}^\top +  (Bg')(Bg')^\top (S_{K'} - S_K)^\top - (Bg)(Bg)^\top S_K^\top)} + \norm*{(T_{K} - T_{K'})\left((Bg')(Bg')^\top S_{K'}^\top \right)}\right) \\
    &\leq \beta_{\gamma} \left(\norm{T_K}\norm{S_K}_2 \norm*{(Bg')(Bg')^\top - (Bg)(Bg)^\top} +  \norm{T_K} \norm*{(Bg')(Bg')^\top (S_{K'} - S_K)^\top} + \norm*{(T_{K} - T_{K'})\left((Bg')(Bg')^\top S_{K'}^\top \right)}\right) \\
    &\labelrel\leq{eq:bound_Sigma_K_pert_i} \beta_{\gamma} \left(\frac{c_3 C}{\sigma_{\min}(Q)} \left(\norm*{Bg'(g')^\top B^\top - Bgg^\top B^\top} +(\norm{B}^2 (c_2)'^2)(2\gamma c_3 \norm{B}_2) \norm{K'-K} \right)+ (4c_3^2)c_5 \norm{K'-K} \norm{B}^2 (c_2')^2  \right) \\
    &\leq \beta_{\gamma} \left(\frac{c_3C}{\sigma_{\min}(Q)} \left(\norm*{B}^2 \norm*{g'(g')^\top - gg^\top} + (\norm{B}^2 (c_2')^2)(2\gamma c_3 \norm{B}_2) \norm{K'-K} \right) + (4c_3^2) c_5 \norm{K'-K} \norm{B}^2 (c_2')^2  \right) \\
    &\leq \beta_{\gamma} \left(\frac{c_3C}{\sigma_{\min}(Q)} \left(\norm*{B}^2 \left(\norm*{g'(g-g')^\top} + \norm*{ (g'- g)(g)^\top}\right) + (\norm{B}^2 (c_2')^2)(2\gamma c_3 \norm{B}_2) \norm{K'-K} \right) + (4c_3^2) c_5 \norm{K'-K} \norm{B}^2 (c_2')^2  \right) \\
    &\leq \beta_{\gamma} \left(\frac{c_3C}{\sigma_{\min}(Q)} \left(\norm*{B}^2 \left(c_2'\norm*{(g-g')^\top} + c_2 \norm*{g'- g}\right) + (\norm{B}^2 (c_2')^2)(2\gamma c_3 \norm{B}_2) \norm{K'-K} \right) + (4c_3^2) c_5 \norm{K'-K} \norm{B}^2 (c_2')^2  \right)
\end{align*}
\normalsize
Above, \eqref{eq:bound_Sigma_K_pert_i} holds by the bound on $\norm*{T_K}$ in Lemma \ref{lemma:bound on T_K norm}, the bound on $\norm{S_K}$ in Lemma \ref{lemma: bounding S_K}, the bound on $\norm*{T_{K'} - T_K}$ from Lemma \ref{Lemma:bound T_K - T_K'}, the bound that $\norm*{g'} \leq c_2'$ from \eqref{eq:bound_on_g'}, and the bound $\norm{S_{K'}}_2 \leq 2 \norm{S_K}$ following from Lemma \ref{lemma:bound S_K' - S_K}.
Then, combining this with Equation \eqref{equation: bound on Phi_k' - Phi_K}, we have that
\small
\begin{align*}
    &\norm*{\Sigma_{K', \tilde{K}'} - \Sigma_{K,\tilde{K}}} \\
    &\leq \norm*{\Phi_{K'} - \Phi_K} + \norm*{\sum_{t=1}^{\infty} \gamma^t \left(T_{K'}\left((B\tilde{K}'x^*)(B\tilde{K}'x^*)^\top S_{K'}^\top \right) -  T_K\left((B\tilde{K}x^*)(B\tilde{K}x^*)^\top S_K^\top \right)\right)} \\
    &\leq c_6 \norm{K'-K} + \beta_{\gamma} \left(\frac{c_3C}{\sigma_{\min}(Q)} \left(\norm*{B}^2 c_2'+ c_2\right)\right) \norm*{g'- g}+ \left(\norm{B}^2 (c_2')^2)(2\gamma c_3 \norm{B}_2) + (4c_3^2) c_5 \norm{B}^2 (c_2')^2  \right) \norm{K'-K},
\end{align*}
\normalsize
which after rearranging yields our result.

\end{proof}

We turn next to bounding $\norm*{q_{K',\tilde{K}'}x^* - q_{K,\tilde{K}}x^*}$. This will be helpful in later bounding $d_{K',\tilde{K}'}x^* - d_{K,\tilde{K}}x^*$, which in turn will assist us in bounding $\nabla J(K',\tilde{K}') - \nabla J (K,\tilde{K}).$

Next, we provide a bound for the quantity $(q_{K',g'} - q_{K,g})$.
\begin{lemma}
\label{lemma:q_K'g' - q_Kg}
Suppose $$(K, \tilde{K}) \in \mathcal{G}_C, \quad \norm{K'-K} \leq \min\{c_4,c_5,1\},\quad \norm*{\tilde{K}' - \tilde{K}} \leq 1.$$ Then,
\begin{align*}
    \norm{q_{K',\tilde{K}'}x^* - q_{K,\tilde{K}}x^*} \leq c_{14} \norm{K'-K} + c_{15} \norm{x^*} \norm{\tilde{K}' - \tilde{K}},
\end{align*}
where 
\[c_{14} := 2\gamma c_3 \norm{B}_2 \norm{Qx^*} + c_{12} + c_{13}, \quad c_{15} := 2c_1 c_3 \norm{R}_2 + \gamma c_3 C \norm{B}_2 \]
and $c_{12}, c_{13} > 0$ are constants depending on $C$ and the system parameters, introduced in the proof later. 
\end{lemma}
\begin{proof}
For notational simplicity, throughout, we denote 
$$q_{K',g'} := q_{K',\tilde{K}'}x^*,\quad q_{K,g} := q_{K,\tilde{K}}x^*. $$
We have
\small
\begin{align}
&q_{K',g'} - q_{K,g} \nonumber \\
&= S_{K'}^\top(-Qx^* + (K')^\top R g' + \gamma (A+BK')^\top P_{K'} Bg') - S_{K}^\top(-Qx^* + K^\top R g + \gamma (A+BK)^\top P_K Bg) \nonumber \\
&= (S_{K'}^\top - S_K^\top)(-Qx^*) + \left(S_{K'}^\top(K')^\top Rg' - S_K^\top K^\top Rg\right) + \left(\gamma S_{K'}^\top (A+BK')^\top P_{K'}Bg' - \gamma S_K^\top (A+BK)^\top P_KBg \right) \label{eq:q_Kg_diff_decomposition}
\end{align}
\normalsize
We treat the three terms on the last line in \eqref{eq:q_Kg_diff_decomposition} separately. First, by Lemma \ref{lemma:bound S_K' - S_K} note that 
\small
\begin{align}
      \norm{(S_{K'}^\top - S_K^\top) (-Qx^*)} \leq (2\gamma \norm{S_K}_2 \norm{B}_2 \norm{K' - K}) \norm{Q}_2 \norm{x^*} \leq 2\gamma c_3 \norm{B}_2 \norm{Q}_2 \norm{x^*} \norm{K' - K}. \label{eq:q_Kg_diff_decomposition_first}
\end{align}
\normalsize 
    Next, observe that 
    \small
\begin{align}
 &\norm{S_{K'}^\top (K')^\top Rg' - S_K^\top K^\top Rg} \nonumber \\
        &= \norm{(S_{K'}^\top - S_K^\top) ((K')^\top Rg') + S_K^\top((K')^\top Rg') - S_K^\top K^\top Rg} \nonumber \\
        &= \norm{(S_{K'}^\top - S_K^\top) ((K')^\top Rg') + S_K^\top((K')^\top Rg' - K^\top Rg) + S_K^\top K^\top R g - S_ K^\top K^\top Rg} \nonumber \\
        &\labelrel\leq{eq:q_Kg_diff_decomposition_second_part1} 2\gamma c_3 \norm{B}_2 \norm{K' - K} ((\norm{K} + 1)\norm{R}_2 \norm{g'} \nonumber \\
        &\quad + 2 c_3\left(\norm*{K^\top R (g' - g) + (K' - K)^\top R g + (K' - K)^\top R (g'- g)}\right) \nonumber \\
        &\labelrel\leq{eq:q_Kg_diff_decomposition_second_part2} 2\gamma c_3 \norm{B}_2 (c_1 + 1)c_2'\norm{R}_2\norm{K'-K} \nonumber \\
        &\quad + 2c_3 \left(\norm{K} \norm{R}_2 \norm{g' - g} + \norm{K'-K} \norm{R}_2 \norm{g} + \norm{K' - K}\norm{R}_2 \norm{g' - g} \right) \nonumber \\
        &\labelrel\leq{eq:q_Kg_diff_decomposition_second_part3} 2\gamma c_3 \norm{B}_2 (c_1 + 1)c_2'\norm{R}_2 \norm{K'-K} \nonumber \\
        &\quad + 2c_3 (c_1 \norm{R}_2 \norm{g'-g} + \norm{K'-K} \norm{R}_2 c_2 + \norm{R}_2 \norm{g'-g}) \nonumber \\
        &= \left(2\gamma c_3 \norm{B}_2 (c_1 + 1)c_2'\norm{R}_2  + \norm{R}_2 c_2 \right) \norm{K' - K}  + (2c_1 c_3 +1) \norm{R}_2 \norm{g' - g} \nonumber \\
        &= \left(2\gamma c_3 \norm{B}_2 (c_1 + 1)c_2'\norm{R}_2  + c_2 \norm{R}_2 \right) \norm{K' - K}  + (2c_1 c_3 +1) \norm{R}_2 \norm{g' - g} \label{eq:q_Kg_diff_decomposition_second}
    \end{align}
    \normalsize 
    Above, \eqref{eq:q_Kg_diff_decomposition_second_part1} follows from the perturbation bound on $\norm{S_{K'} - S_K}_2$ in Lemma \ref{lemma:bound S_K' - S_K}, the bound on $\norm{S_K}_2$ in Lemma \ref{lemma: bounding S_K}, and our assumption that $\norm{K' - K} \leq 1$. Meanwhile, \eqref{eq:q_Kg_diff_decomposition_second_part2} and \eqref{eq:q_Kg_diff_decomposition_second_part3} both hold due to our bounds on $K$  Lemma \ref{lemma:bound K and g}, our assumption that $\norm{K' - K} \leq 1$, as well as the bound on $\norm{g'}$ on Line \ref{eq:bound_on_g'}. 
    
    We next turn to the last summand in \eqref{eq:q_Kg_diff_decomposition}, namely $$\left(\gamma S_{K'}^\top (A+BK')^\top P_{K'}Bg' - \gamma S_K^\top  (A+BK)^\top P_KBg \right).$$ 
    By a similar algebraic manipulation to the immediately preceding set of derivations, we find that
    \begin{align*}
        &\norm*{\gamma S_{K'}^\top (A+BK')^\top P_{K'}Bg' - \gamma S_K^\top  (A+BK)^\top P_KBg} \\
        &= \norm*{(S_{K'}^\top - S_K^\top) \gamma (A + BK')^\top P_{K'} Bg' + S_K^\top(\gamma (A + BK')^\top P_{K'} Bg' - \gamma (A + BK)^\top P_{K} Bg)} \\
        &\leq \gamma \norm{S_{K'}^\top - S_K^\top} \norm{(A + BK')^\top P_{K'} Bg'}_2 + \norm{S_K}_2 \norm{ \gamma (A + BK')^\top P_{K'} Bg' - \gamma (A + BK)^\top P_{K} Bg}.
    \end{align*}
    We first handle $\gamma \norm{S_{K'} - S_K} \norm{(A + BK')^\top P_{K'} Bg'}_2$. 
    This involves bounding $\sqrt{\gamma}\norm{(A + BK')^\top P_{K'} Bg'}_2$. To this end, observe that
    \begin{align*}
        &\sqrt{\gamma}\norm{(A + BK')^\top P_{K'} Bg'}_2 \\
        &\labelrel\leq{eq:q_Kg_diff_decomposition_third_1_part1}\norm{P_{K'}}_2 \norm{B}_2 \norm{g'} \\
        &= \norm{P_K + P_{K'} - P_K}_2 \norm{B}_2 \norm{g'} \\
        &\labelrel\leq{eq:q_Kg_diff_decomposition_third_1_part2} (\norm{P_K}_2 + c_{11} \norm{K'-K}) \norm{B}_2 (c_2') \\
        &\labelrel\leq{eq:q_Kg_diff_decomposition_third_1_part3} (C + c_{11})c_2' \norm{B}_2 
    \end{align*}
    Above, \eqref{eq:q_Kg_diff_decomposition_third_1_part1} holds due to our choice of $K'$ such that $\sqrt{\rho} \norm{A+BK'} \leq 1$, while \eqref{eq:q_Kg_diff_decomposition_third_1_part2} holds due to Lemma \ref{Lemma:bound P_K' - P_K and f(K') - f(K)} which bounds $\norm{P_K + P_{K'} - P_K}_2$ as well as the fact that $\norm{g'} \leq c_2'$, which we have seen earlier. Finally, \eqref{eq:q_Kg_diff_decomposition_third_1_part3} comes from the bound on $\norm{P_K}_2$ in Lemma \ref{Lemma:bound Phi_k and P_K}. 
    By Lemma \ref{lemma:bound S_K' - S_K} which bounds $\norm{S_{K'} - S_K}$, this gives us 
    \begin{align}
        \gamma \norm{S_{K'} - S_K} \norm{(A + BK')^\top P_{K'} Bg'} &\leq \sqrt{\gamma}(C + c_{11})c_2' \norm{B}_2 \norm{S_{K'} - S_K} \nonumber \\
        &\leq \sqrt{\gamma}(C + c_{11})c_2' \norm{B}_2 \left( 2\gamma \norm{S_K}_2 \norm{B}_2 \norm{K' - K}\right) \nonumber \\
        &\leq 2 \gamma^{3/2}  (C + c_{11})c_2' \norm{B}_2^2  \norm{K'-K}. \label{eq:q_Kg_diff_decomposition_third_1}
    \end{align}
    Next, we handle $\norm{S_K}_2 \norm{ \gamma (A + BK')^\top P_{K'} Bg' - \gamma (A + BK)^\top P_{K} Bg}$. Following a similar application of relevant norm and perturbation bounds, we find that
    \small
    \begin{align}
        &\norm{S_K}_2 \norm{ \gamma (A + BK')^\top P_{K'} Bg' - \gamma (A + BK)^\top P_{K} Bg}_2 \nonumber \\
        &= \norm{S_K}_2 \norm*{ \gamma((A+BK')^\top - (A+BK)^\top) P_{K'}Bg' + \gamma (A+BK)^\top (P_{K'}Bg' - P_KBg)} \nonumber \\
        &= \norm{S_K}_2 \norm*{ \gamma((A+BK')^\top - (A+BK)^\top)} \norm*{P_{K'}Bg'} + \gamma \norm*{ (A+BK)^\top \left((P_{K'} - P_K)Bg' + P_KB(g'-g) \right)} \nonumber \\
        &\leq \gamma c_3 \norm{B}_2^2 \norm{K'-K} (C + c_{11})c_2' \norm{B}_2 + \sqrt{\gamma}c_3 \left(\norm{P_{K'} - P_K} \norm{B}_2 c_2'\right) + \gamma \norm{P_K}_2 \norm{B}_2 \norm{g'-g} \nonumber \\
        &\leq \gamma c_3 \norm{B}_2^2 (C + c_{11})c_2' \norm{B}_2 \norm{K'-K} + \sqrt{\gamma}c_3 \norm{B}_2 c_2'c_{11}\norm{K'-K} + C \norm{B}_2 \norm{g'-g} \nonumber \\
        &\leq c_3(\gamma \norm{B}_2^2 (C + c_{11})c_2' \norm{B}_2 + \sqrt{\gamma} \norm{B}_2 c_2' c_{11}) \norm{K'-K} + \gamma c_3 C \norm{B}_2 \norm{g'-g}. \label{eq:q_Kg_diff_decomposition_third_2}
    \end{align}
    \normalsize
    Define now 
    \small
\begin{align*}
        &c_{12} := 2\gamma c_3 \norm{B}_2 (c_1 + 1)c_2'\norm{R}_2 + c_2 \norm{R}_2, \\
        & c_{13} := 2 \gamma^{3/2}  (C + c_{11})c_2' \norm{B}_2^2 + c_3(\gamma \norm{B}_2^2 (C + c_{11}) c_2' \norm{B}_2 + \sqrt{\gamma} \norm{B}_2 c_2' c_{11}).
    \end{align*}
    \normalsize
     Then, combining the bounds in \eqref{eq:q_Kg_diff_decomposition_first}, \eqref{eq:q_Kg_diff_decomposition_second}, \eqref{eq:q_Kg_diff_decomposition_third_1}, and \eqref{eq:q_Kg_diff_decomposition_third_2}, we find that 
    \begin{align*}
        \norm{q_{K',g'} - q_{K,g}} &\leq (2\gamma c_3 \norm{B}_2 \norm{Qx^*} + c_{12} + c_{13}) \norm{K'-K} + \left((2c_1 c_3 +1) \norm{R}_2 + \gamma c_3 C \norm{B}_2\right) \norm{g'-g}.
    \end{align*}
\end{proof}

We next bound the quantities $C_{K'} - C_K$ and $d_{K',\tilde{K}'}x^* - d_{K,\tilde{K}}x^*.$ 
\begin{lemma}
\label{lemma:bound C_K - C_k' and d_K'g' - d_Kg}
Suppose $$(K,\tilde{K}) \in \mathcal{G}_C, \quad \norm{K' - K} \leq \min\{c_4,c_5,1\}, \quad \norm{\tilde{K}' - \tilde{K}} \leq 1.$$ Then, 
\begin{align*}
    & \norm{C_{K'} - C_K} \leq c_{16} \norm{K'-K}  \\
    & \norm{d_{K',\tilde{K}'}x^* - d_{K,\tilde{K}}x^*} \leq c_{17} \norm{K'-K} + c_{18} \norm{x^*} \norm{\tilde{K}' - \tilde{K}},
\end{align*} 
where 
\begin{align*}
    &c_{16} = \norm{R}_2 + \gamma \norm{B}_2 [c_{11}+ \norm{B}_2 C)] \\
    &c_{17} = \gamma \norm{B}_2^2 c_2' c_{11} + \gamma \norm{B}_2 c_{14} \\
    &c_{18} = \norm{R}_2 + \gamma \norm{B}_2^2 C + \gamma \norm{B}_2 c_{15}.
\end{align*}
\end{lemma}
\begin{proof}
For notational simplicity, we denote
$$d_{K',g'} := d_{K',\tilde{K}'}x^*, \quad d_{K,g} := d_{K,\tilde{K}}x^*.$$
We first handle $C_{K'} - C_K$. Using norm bounds on $K,g$, perturbation bounds on $\norm{P_{K'} - P_K}$, and our assumptions on $K'$, we find that 
\begin{align*}
    &\norm{C_{K'} - C_K} \\
    &\leq \norm{RK' + \gamma B^\top P_{K'} (B K' + A) - (RK + \gamma B^\top P_K (BK + A))} \\
    &\leq \norm{R}_2 \norm{K'-K} + \gamma \norm{B}_2 \norm{(P_{K'} - P_K) (A + BK') + P_KB(K' - K)} \\
    &\leq \norm{R}_2 \norm{K'-K} + \gamma \norm{B}_2 [c_{11}\norm{K'-K} + \norm{B}_2 C \norm{K'-K}] \\
    &= (\norm{R}_2 + \gamma \norm{B}_2 [c_{11}+ \norm{B}_2 C)] )\norm{K'-K}.
\end{align*}
Meanwhile, using norm bounds on $K,g$, perturbation bounds on $\norm{P_{K'} - P_K}$ and $\norm*{q_{K',g'} - q_{K,g}}$, we find that
\begin{align*}
    &\norm{d_{K',g'} - d_{K,g}} \\
    &= \norm{Rg' + \gamma B^\top P_{K'} Bg' + \gamma B^\top q_{K',g'} - (Rg + \gamma B^\top P_{K} Bg + \gamma B^\top q_{K,g})} \\
    &\leq \norm{R}_2 \norm{g'-g} + \gamma \norm{B}_2 \norm{(P_{K'} - P_K) Bg' + P_K B(g'-g)} + \gamma \norm{B}_2 \norm{q_{K',g'} - q_{K,g}} \\
    &\leq \norm{R}_2 \norm{g'-g} + \gamma \norm{B}_2 ( \norm{B}_2 c_2' c_{11}\norm{K'-K} + C \norm{B}_2 \norm{g'-g})\\
    &\quad + \gamma \norm{B}_2 (c_{14} \norm{K'-K} + c_{15}\norm{g'-g}) \\
    &\leq (\norm{R}_2 + \gamma \norm{B}_2^2 C + \gamma \norm{B}_2 c_{15})\norm{g'-g} +  (\gamma \norm{B}_2^2 c_2' c_{11} + \gamma \norm{B}_2 c_{14}) \norm{K'-K}. 
\end{align*}
\end{proof}

\subsubsection{Local Lipschitzness of $\nabla J$}
We are now finally ready to combine the various norm and perturbation bounds to show that $J$ has locally Lipschitz gradient.

\begin{proposition}[$\nabla J(K,\tilde{K})$ is locally Lipschitz]
Suppose $$(K,\tilde{K}) \in \mathcal{G}_C, \quad \norm{K' - K} \leq \min\{c_4,c_5,1\}, \quad \norm{\tilde{K}' - \tilde{K}} \leq 1.$$ Then, there exists $L > 0$ depending only on $(A,B,Q,R,\Sigma,\gamma,x^*,C)$ such that 
\begin{align*}
    &\norm{\nabla J(K',\tilde{K}') - \nabla J(K,\tilde{K})} \leq L \left(\norm*{K'-K} + \norm*{\tilde{K}' - \tilde{K}} \right). 
\end{align*}
\end{proposition}
\begin{proof}

We seek to bound $\norm*{\nabla J(K',\tilde{K}') - \nabla J(K,\tilde{K})}^2$. We consider the derivative terms in $K$ and $\tilde{K}$ separately. We first tackle $\norm{\nabla_K J(K',\tilde{K}') - \nabla_K J(K,\tilde{K})}$. From Lemma \ref{lemma:gradient of J}, we know that 
\begin{align*}
    \nabla_K J(K,\tilde{K}) &= 2C_K \Sigma_{K,\tilde{K}} + 2d_{K,\tilde{K}}x^* \rho_{K,\tilde{K}}^\top,
\end{align*}
Hence,
\begin{align*}
    &\norm*{\nabla_K J(K',\tilde{K}') - \nabla_K J(K,\tilde{K})} \\
    &= \norm*{2C_{K'} \Sigma_{K',\tilde{K}'} + 2d_{K',\tilde{K}'}x^* \rho_{K',\tilde{K}'}^\top - \left(2C_K \Sigma_{K,\tilde{K}} + 2d_{K,\tilde{K}}x^* \rho_{K,\tilde{K}}^\top \right)} \\
    &\leq 2 \norm*{C_{K'} \Sigma_{K,\tilde{K}} + C_{K'} (\Sigma_{K',\tilde{K}'}-\Sigma_{K,\tilde{K}}) - C_K \Sigma_{K,\tilde{K}}} \\
    &\quad + 2 \norm*{d_{K',\tilde{K}'}x^* \rho_{K,\tilde{K}}^\top + d_{K',\tilde{K}'}x^* (\rho_{K',\tilde{K}'}^\top - \rho_{K,\tilde{K}}^\top) -d_{K,\tilde{K}}x^* \rho_{K,\tilde{K}}^\top } \\
    &\leq 2 \norm*{(C_{K'} - C_K)}\norm*{\Sigma_{K,\tilde{K}}} + 2 \norm*{C_{K'}} \norm*{\Sigma_{K',\tilde{K}'}-\Sigma_{K,\tilde{K}}} \\
    &\quad + 2 \norm*{d_{K',\tilde{K}'}x^* - d_{K,\tilde{K}}x^*} \norm*{\rho_{K,\tilde{K}}} + 2 \norm*{d_{K',\tilde{K}'}x^*} \norm*{\rho_{K',\tilde{K}'} - \rho_{K,\tilde{K}}}
\end{align*}
By the perturbation bounds in Lemma \ref{lemma:bound C_K - C_k' and d_K'g' - d_Kg} as well as Lemma \ref{lemma: bound Sigma_K'Ktilde' - Sigma_KKtilde} and Lemma \ref{lemma: bound rho_K'Ktilde' - rho_KKtilde}, as well as bounds on the norms of $\Sigma_{K,\tilde{K}}, \rho_{K,\tilde{K}}x^*, C_K, d_{K,\tilde{K}}x^*$ we derived in Lemma \ref{lemma:bound on Sigma_K,Ktilde}, Lemma \ref{lemma:bound on rho_KL}, and Lemma \ref{lemma:bound_C_K_d_KKtilde} respectively,  it follows that there exists $c'' > 0$ depending only on $(A,B,Q,R,\Sigma,\gamma,x^*,C)$ such that 
\begin{align*}
    \norm*{\nabla_K J(K',\tilde{K}') - \nabla_K J(K,\tilde{K})} \leq c'' \left(\norm*{K'-K} + \norm*{\tilde{K}' - \tilde{K}} \right).
\end{align*}

It remains to bound $\norm*{\nabla_{\tilde{K} } J(K',\tilde{K}') - \nabla_{\tilde{K} } J(K,\tilde{K})}$. To this end, recall $\nabla_{\tilde{K}}J$ which we computed in Lemma \ref{lemma:gradient of J}, giving us
\begin{align*}
    &\norm*{\nabla_{\tilde{K} } J(K',\tilde{K}') - \nabla_{\tilde{K} } J(K,\tilde{K})}\\
    &= \norm*{2(C_{K'} \rho_{K',\tilde{K}'}x^* + d_{K',\tilde{K}'}x^* (x^*)^\top) -2(C_{K} \rho_{K,\tilde{K}}x^* + d_{K,\tilde{K}}x^* (x^*)^\top)} \\
    &\leq 2 \norm*{C_{K'} \rho_{K,\tilde{K}}x^* + C_{K'} (\rho_{K',\tilde{K}'}-\rho_{K,\tilde{K}})x^* - C_K \rho_{K,\tilde{K}}x^*}  \\
    &\quad + 2 \norm*{d_{K',\tilde{K}'}x^* (x^*)^\top - d_{K,\tilde{K}}x^* (x^*)^\top}  \\
    &\leq 2 \norm*{(C_{K'} - C_K)}\norm*{\rho_{K,\tilde{K}}x^*} + 2 \norm*{C_{K'}} \norm*{\rho_{K',\tilde{K}'}x^*-\rho_{K,\tilde{K}}x^*} \\
    &\quad + 2 \norm*{d_{K',\tilde{K}'}x^* (x^*)^\top - d_{K,\tilde{K}}x^* (x^*)^\top} .
\end{align*}
By the perturbation bounds in Lemma \ref{lemma:bound C_K - C_k' and d_K'g' - d_Kg} and Lemma \ref{lemma: bound rho_K'Ktilde' - rho_KKtilde}, as well as bounds on the norms of $\rho_{K,\tilde{K}}, C_K, d_{K,\tilde{K}}x^*$ we derived in Lemma \ref{lemma:bound on rho_KL} and Lemma \ref{lemma:bound_C_K_d_KKtilde} respectively, it follows that there exists $c''' > 0$ depending only on $(A,B,Q,R,\Sigma,\gamma,x^*,C)$ such that 
\begin{align*}
    \norm*{\nabla_{\tilde{K} } J(K',\tilde{K}') - \nabla_{\tilde{K} } J(K,\tilde{K})} \leq c''' \left(\norm*{K'-K} + \norm*{\tilde{K}' - \tilde{K}} \right).
\end{align*}
The statement in the lemma then follows by picking a $L$ depending only on the constants $c,c', c'', c'''$.
\end{proof}

\subsubsection{Local Lipschitzness of $J$}
We now state and prove the local Lipschitz property of $J(K,\tilde{K})$. 
\begin{proposition}[Local Lipschitzness of $J(K,\tilde{K})$]
Suppose $$(K,\tilde{K}) \in \mathcal{G}_C, \quad \norm{K' - K} \leq \min\{c_4,c_5,1\}, \quad \norm{\tilde{K}' - \tilde{K}} \leq 1.$$ Then, there exists $\lambda > 0$ depending only on $(A,B,Q,R,\Sigma,\gamma,x^*,C)$ such that 
\begin{align*}
    \abs{J(K',\tilde{K}') - J(K,\tilde{K})} \leq \lambda \left(\norm*{K'-K} + \norm*{\tilde{K}' - \tilde{K}} \right).
\end{align*}
\end{proposition}
\begin{proof}
Recall from Line \eqref{eq:value_diff_eq_intermediate_for_smoothness} in our computations in the outline that
\begin{align*}
    &J(K',\tilde{K}') - J(K,\tilde{K}) \nonumber \\
    &= 2 \tr((K'-K)^\top (C_K \Sigma_{K',\tilde{K}'}  + d_{K,\tilde{K}}x^* \rho_{K',\tilde{K}'}^\top)) + 2\tr((\tilde{K}' - \tilde{K})^\top (C_K \rho_{K',\tilde{K}'} + d_{K,\tilde{K}}x^*)(x^*)^\top) \\
    &\quad + 2 \tr((K'-K)^\top (R + \gamma B^\top P_K B)(\tilde{K}' - \tilde{K})x^* \rho_{K',\tilde{K}'}^\top) \\
    &\quad + \tr(\Sigma_{K',\tilde{K}'}(K'-K)^\top (R + \gamma B^\top P_K B)(K'-K)) + \tr((\tilde{K}' - \tilde{K})^\top (R + \gamma B^\top P_K B)(\tilde{K}' - \tilde{K})x^*(x^*)^\top). 
\end{align*}
Since we have bounds on the terms
\begin{align*}
    C_K, \Sigma_{K, \tilde{K}}, d_{K,\tilde{K}}x^*, \rho_{K,\tilde{K}}x^*, P_K, \tilde{K}x^*, 
\end{align*}
from Lemmas \eqref{lemma:bound K and g}, \eqref{Lemma:bound Phi_k and P_K}, \eqref{lemma:bound on rho_KL},\eqref{lemma:bound on Sigma_K,Ktilde} , and \eqref{lemma:bound C_K - C_k' and d_K'g' - d_Kg}, by our choice of $K'$ and $\tilde{K}'$, applying the perturbation bounds in Lemmas \eqref{lemma: bound rho_K'Ktilde' - rho_KKtilde} and \eqref{lemma: bound Sigma_K'Ktilde' - Sigma_KKtilde}, we have bounds on 
\begin{align*}
    \Sigma_{K', \tilde{K}'}, \rho_{K,\tilde{K}}
\end{align*}
as well. Therefore, it follows that there exists $\lambda > 0$ depending only on $(A,B,Q,R,\Sigma,\gamma,x^*,C)$ such that
\begin{align*}
    \abs{J(K',\tilde{K}') - J(K,\tilde{K})} \leq \lambda \left(\norm*{K'-K} + \norm*{\tilde{K}' - \tilde{K}} \right).
\end{align*}
\end{proof}

Finally, we need to show a local Lipschitz property for the sample cost variant $J(K,\tilde{K}; x_0))$ starting from an initial state $x_0$. To facilitate technical analysis, we work with the assumption $\norm{x_0} \leq C_n$ for some $C_n > 0$ (cf. \cite{malik2019derivative}).
\begin{proposition}[Local Lipschitzness of sample cost $J(K,\tilde{K};x_0)$]
\label{proposition: sample cost locally Lipschitz}
Suppose $$(K,\tilde{K}) \in \mathcal{G}_C, \quad \norm{K' - K} \leq \min\{c_4,c_5,1\}, \quad \norm{\tilde{K}' - \tilde{K}} \leq 1, \quad \norm{x_0} \leq C_n.$$ Then, there exists $\lambda > 0$ depending only on $(A,B,Q,R,\Sigma,\gamma,x^*,C,C_n)$ such that 
\begin{align*}
    \abs{J(K',\tilde{K}';x_0) - J(K,\tilde{K};x_0)} \leq \lambda \left(\norm*{K'-K} + \norm*{\tilde{K}' - \tilde{K}} \right).
\end{align*}
\end{proposition}
\begin{proof}
We note that by Lemma \ref{lemma:value function of J(K,g)}, recalling that $g := \tilde{K}x^*$,
\begin{align*}
    J(K,\tilde{K};x_0) = x_0^\top P_Kx_0 + x_0^\top q_{K,g} + r_{K,g}.
\end{align*}
Since we already have perturbation bounds on $\norm{P_K - P_{K'}}$ and $q_{K,g}$ from Lemma \ref{Lemma:bound P_K' - P_K and f(K') - f(K)} and \ref{lemma:q_K'g' - q_Kg} respectively for $K'$ and $\tilde{K}'$ selected as in the assumptions, it remains for us to bound $r_{K',g'} - r_{K',g'}$. From Lemma \ref{lemma:value function of J(K,g)}, we have that
\begin{align*}
    r_{K,g} = \frac{1}{1-\gamma} \left((x^*)^\top Qx^* + g^\top Rg + \gamma \left(g^\top B^\top P_K Bg + 2g^\top B^\top q_{K,g} \right) \right).
\end{align*}
By the perturbation techniques used earlier, it is clear from the form of $r_{K,g}$ that there exists $c > 0$ depending only on $(A,B,Q,R,\Sigma,\gamma,x^*,C)$ such that 
\begin{align*}
    \abs{r_{K',g'} - r_{K,g}} \leq c (\norm{K'-K} + \norm{g'-g}).
\end{align*}
Therefore, since we also assumed an uniform bound $C_n$ on $\norm{x_0}$, there exists a $\lambda > 0$ depending only on $(A,B,Q,R,\Sigma,\gamma,x^*,C,C_n)$ such that 
\begin{align*}
    \abs{J(K',\tilde{K}';x_0) - J(K,\tilde{K};x_0)} \leq \lambda  (\norm{K'-K} + \norm{g'-g}).
\end{align*}
\end{proof}

\section{Analysis of main results}

\subsection{Results from zeroth-order optimization}

The following results from zeroth-order optimization are useful. The two-point estimators we use in the algorithm can be defined as follows:
\begin{align*}
    & z_{r}((K,\tilde{K}),x_0^i,\delta^i) := \frac{d}{2r}\left(J((K,\tilde{K}) + r\delta^i; x_0^i) - J(((K,\tilde{K}) - r\delta^i; (x_0)^i)\right)\delta^i,
\end{align*}
where $i \in [m]$ denotes the indexing in the mini-batch, and $d = 2nk$ (i.e. dimension of $K$ plus dimension of $\tilde{K}$). When the context is clear, we may some times omit the index $i$. 
An essential quantity is the following smoothed version of $J$:
\begin{align*}
    J_r(K,\tilde{K}) = \bbE_{\bar{\delta} \sim \bbB^{d}}\left[J(K,\tilde{K}) + r\bar{\delta} \right].
\end{align*}

This next lemma shows that the zeroth-order estimator $z_{r}$ is an unbiased estimate of $\nabla J_r(K,\tilde{K})$, which in turn is close to $\nabla J(K,\tilde{K})$ when $r$ is small. 

\begin{proposition}[Properties of two-point estimator]
\label{proposition: properties of two-point estimator}
Suppose that $(K,\tilde{K}) \in \mathcal{G}_{10J_0}$. Suppose $r \leq \rho,$ the local radius of smoothness defined earlier.
Then, the following are true:
\begin{enumerate}[(i)]
    \item The two-point estimator of $J$, $z_{r}(\cdot)$ is an unbiased estimator of $\nabla J_r(\cdot)$.
\begin{align*}
    \mathbb{E}[z_{r}((K,\tilde{K}),x_0,\delta)] =
    \nabla J_r(K,\tilde{K}).
\end{align*}
\item The gradient of the smoothed version of $J$ is close to the gradient of $J$.
\[\norm{\nabla J_r(K, \tilde{K}) - \nabla J (K,\tilde{K})} \leq rL.\]
\item There exists a maximum bound on the size of each estimate:
\begin{align*}
    \sup_{\substack{\hat{K} \in \mathcal{G}_{10J_0} \\ x_0,\delta}} \left\{\norm*{z_r(\hat{K}; x_0,\delta)}\right\} \leq d\lambda,
\end{align*}
where $\lambda$ is the local Lipschitz parameter defined earlier. 
\item Finally, the variance of the estimator is also bounded.
\begin{align*}
    \sup_{\hat{K} \in \mathcal{G}_{10J_0}}\left\{\!\bbE\! \left[\norm*{z_r(\hat{K};x_0,\delta) - \bbE \left[z_r(\hat{K};x_0,\delta)  \mid \hat{K} \right] }^2\right]\right\} \leq d \lambda^2 
\end{align*}
\end{enumerate}

Therefore, for the two-point estimator, we can set
\[Z_{\infty} := d \lambda, \quad Z_2 := d \lambda^2 \]
\end{proposition}
\begin{proof}
The first result is a consequence of Lemma 2.1 in \cite{flaxman2004online}, and the symmetry of the distribution of $u$, which we assume to be $\bbS^{d-1}$. Result (ii) follows directly from Lemma 6(b) in \cite{malik2019derivative}. Results (iii), (iv) are consequences of the argument in Section 4.3 of \cite{malik2019derivative}. We note in particular that result (iv), which utilizes the argument made in \cite{shamir2017optimal} to bound the variance of the symmetric two-point estimator, requires the fact that the sample cost is also Lipschitz (uniformly over a local radius). 
\end{proof}



\subsection{Proof of Lemma \ref{lemma: per-iteration-change-in-optimality gap-first-time}: reduction in optimality gap}
\label{appendix:optimality-gap}

Recall the definition
\begin{align*}
    \Delta_t := J(K_t,\tilde{K}_t) - J^{\mathrm{opt}},
\end{align*}
and the statement of Lemma \ref{lemma: per-iteration-change-in-optimality gap-first-time}.
\perIterOptGap*
\begin{proof}
By choosing $\eta \leq \frac{\rho}{Z_{\infty}}$, we find that
\begin{align*}
    \norm*{(K_{t+1}, \tilde{K}_{t+1}) - (K_t,\tilde{K}_t)}^2 &=
    \norm*{(K_t,\tilde{K}_t) - \eta (z_t) (K_t, \tilde{K}_t) - (K_t, \tilde{K}_t)}^2 \\
    &\leq \eta^2 (Z_{\infty}^2) \leq \rho^2.
\end{align*}
Since $(K_t,\tilde{K}_t) \in \mathcal{G}_{10J_0}$, by local smoothness, and taking conditional expectation on $\mathcal{F}_t$ (which we denote by $\mathbb{E}^t$), it follows that
\begin{align*}
    \mathbb{E}^t \left[J(K_{t+1}, \tilde{K}_{t+1})\right] &\labelrel\leq{eq:improvement_smoothness} J(K_t,\tilde{K}_t) - \eta \mathbb{E}^t \brac*{\nabla J(K_t, \tilde{K}_t), \frac{1}{m} \sum_{i=1}^m (z_t^i)(K_t, \tilde{K}_t)} + \frac{L}{2} \mathbb{E}^t \norm*{\eta \frac{1}{m} \sum_{i=1}^m (z_t^i)(K_t, \tilde{K}_t)}^2 \\
    &\labelrel={eq:near_unbiased_ZO_estimate} J(K_t,\tilde{K}_t) - \eta \brac*{\nabla J(K_t, \tilde{K}_t), \nabla J_r(K_t, \tilde{K}_t)} + \frac{L\eta^2}{2m^2} \mathbb{E}^t  \norm*{\sum_{i=1}^m (z_t^i)(K_t, \tilde{K}_t)}^2.
\end{align*}
Above, \eqref{eq:improvement_smoothness} is a consequence of the smoothness result in Proposition \ref{proposition: local smoothness}, while \eqref{eq:near_unbiased_ZO_estimate} is a consequence of Proposition \ref{proposition: properties of two-point estimator} (i). 

We first handle the term $- \eta \brac*{\nabla J(K_t, \tilde{K}_t), \nabla J_r(K_t, \tilde{K}_t)}$.
Observe that
\begin{align}
    & -\eta \brac*{\nabla J(K_t, \tilde{K}_t),\nabla J_r(K_t, \tilde{K}_t)} \nonumber \\
    &= -\eta \brac*{\nabla J(K_t, \tilde{K}_t), \nabla J(K_t, \tilde{K}_t)} - \eta \brac*{\nabla J(K_t, \tilde{K}_t),  \nabla J_r(K_t, \tilde{K}_t) - \nabla J(K_t, \tilde{K}_t)} \nonumber \\
    &\labelrel={eq:2ab <= a^2+b^2} - \eta \norm*{\nabla J (K_t, \tilde{K}_t)}^2 + \eta\ \left( \frac{\norm*{\nabla J(K_t, \tilde{K}_t)}^2}{2} + \frac{\norm*{\nabla J_r(K_t, \tilde{K}_t) - \nabla J(K_t, \tilde{K}_t)}^2}{2}\right) \nonumber \\
    &\leq - \frac{\eta}{2} \norm*{\nabla J (K_t, \tilde{K}_t)}^2 + \frac{\eta}{2} \norm*{\nabla J_r(K_t, \tilde{K}_t) - \nabla J(K_t, \tilde{K}_t)}^2 \nonumber \\
    &\labelrel\leq{eq:less_than_rL} - \frac{\eta}{2} \norm*{\nabla J (K_t, \tilde{K}_t)}^2 + \frac{\eta}{2} L^2 r^2 \label{line:grad_J grad_Jr inner product}
\end{align}
Above, \eqref{eq:2ab <= a^2+b^2} is a consequence of the fact that for any vectors $a,b$, $2\brac{a,b} \leq \norm{a}^2 + {b}^2$. Meanwhile, to obtain \eqref{eq:less_than_rL}, we used Proposition \ref{proposition: properties of two-point estimator} (ii), utilizing our choice of $r \leq \rho$.

We turn our attention now to $\mathbb{E}^t \norm*{\frac{1}{m} \sum_{i=1}^m z_t^i(K_t,\tilde{K}_t)}^2$. For notational convenience we use $\var(\cdot)$ to denote $\mathbb{E}[\norm*{\cdot - \bbE[\cdot]}^2]$. Then, 
\begin{align}
  &\mathbb{E}^t \norm*{\frac{1}{m} \sum_{i=1}^m (z_t^i) (K_t,\tilde{K}_t)}^2 \nonumber \\
  &\labelrel={eq:u_i indep}  \frac{1}{m^2} \sum_{i=1}^m \var((z_t^i)) + \norm*{\nabla J_r (K_t, \tilde{K}_t)}^2 \nonumber \\
  &\labelrel={eq:(a+b)^2 <= 2a^2 + 2b^2} \frac{1}{m^2} \sum_{i=1}^m \var((z_t^i)) + 2 \norm*{\nabla J(K_t, \tilde{K}_t)}^2  + 2\norm*{\left(\nabla J_r (K_t,\tilde{K}_t) - \nabla J(K_t, \tilde{K}_t)\right)}^2 \nonumber \\
  &\labelrel\leq{eq:variance_last_line} \frac{1}{m} Z_{2} + 2 \norm*{\nabla J(K_t,\tilde{K}_t)}^2 + 2L^2r^2. \label{line:avg_variance_bound}
\end{align}
For \eqref{eq:u_i indep}, we utilized independence of the random perturbation $\{u_t^i\}$ across $i \in [m]$. For \eqref{eq:(a+b)^2 <= 2a^2 + 2b^2}, we used the fact that for any vectors $a,b$, $\norm{a}^2 \leq 2\norm{a-b}^2 +  2\norm{b}^2$.

For \eqref{eq:variance_last_line}, denoting $\hat{K} := (K,\tilde{K})$, we utilized the definition 
$$Z_{2} := \!\sup_{\hat{K} \in \mathcal{G}_{10J_0}}\left\{\!\bbE\! \left[\norm*{z_r(\hat{K};x_0,\delta) - \bbE \left[z_r(\hat{K};x_0,\delta)  \mid \hat{K} \right] }^2\right]\right\}\!,$$ 
as well as Proposition \ref{proposition: properties of two-point estimator}(ii).

Putting everything together, we find that
\begin{align*}
    &\mathbb{E}^t \left[J(K_{t+1}, \tilde{K}_{t+1})\right] - J^{\mathrm{opt}} \\
    &\labelrel\leq{eq:improvement_first_part_bound} J(K_t,\tilde{K}_t) - \frac{\eta}{2} \norm*{\nabla J (K_t, \tilde{K}_t)}^2 + \frac{\eta}{2} L^2 r^2  + \frac{L\eta^2}{2}\mathbb{E}^t\left( \norm*{\frac{1}{m} \sum_{i=1}^m (z_t^i)(K_t,\tilde{K}_t)}^2 \right) - J^{\mathrm{opt}}\\
    &\labelrel\leq{eq:improvement_second_part_bound} \Delta_t - \frac{\eta}{2} \norm*{\nabla J (K_t, \tilde{K}_t)}^2 + \frac{\eta}{2} L^2 r^2 + \frac{L\eta^2}{2} \left(\frac{Z_{2}}{m} + 2 \norm*{\nabla J(K_t,\tilde{K}_t)}^2 +  2L^2r^2 \right) \\
    &\leq \Delta_t - \left(\frac{\eta}{2} - L\eta^2 \right) \norm*{\nabla J(K_t,\tilde{K}_t)}^2 + (L^2\eta + L^3\eta^2 )r^2 + \frac{L\eta^2}{2m}(Z_2) \\
    &\labelrel\leq{improvement_final_PL} \left(1 - \frac{\eta\mu}{4} \right)\Delta_t + 2\eta L^2 r^2 +  \frac{L\eta^2}{2m}(Z_2).
\end{align*}
Above, \eqref{eq:improvement_first_part_bound} holds by the bound on Line \ref{line:grad_J grad_Jr inner product}, while \eqref{eq:improvement_first_part_bound} holds by the bound on Line \ref{line:avg_variance_bound}. Finally, \eqref{improvement_final_PL} holds by picking $\eta L \leq \frac{1}{4}$, so that $\frac{\eta}{2} - \eta^2 L \geq \frac{\eta}{4}$, as well as gradient dominance in Proposition \ref{proposition: gradient dominant single agent}.

Continuing, using the choice 
\begin{align*}
    r \leq \sqrt{\frac{\ep\mu}{240L^2}},
\end{align*}
we get that
\begin{align*}
    \mathbb{E}^t [\Delta_{t+1}] \leq \left(1 - \frac{\eta\mu}{4}\right) \Delta_t + \frac{\eta^2L}{2}\left(\frac{Z_2}{m} \right) + \frac{\eta\mu}{120}\ep.
\end{align*}


\end{proof}

\subsection{Corollary 1: sample complexity}
\label{appendix:sample-complexity}

\begin{proof}[Proof of Corollary \ref{corollary:sample_complexity_2pt}]
From Theorem \ref{theorem:lqr_tracking_algorithm_convergence}, by choosing 
\begin{align*}
    &\eta \leq \min \left\{\frac{m \ep \mu }{240L \left(Z_2 \right)}, \frac{1}{4L},  \frac{\rho}{Z_{\infty}}\right\}, r \leq \min \left\{ \sqrt{\frac{\ep \mu}{240 L^2}}, \rho \right\},
\end{align*}
we know that it takes 
$$ T = \frac{4}{\mu \eta} \log(120\Delta_0/\ep)$$
steps so that the final optimality gap satisfies $\Delta_T \leq \ep$. Assuming that the iterates of the algorithm stay in $\mathcal{G}_{10J_0}$, which we know by Proposition \ref{proposition:stability-first-time} happens with probability at least 0.8, by Proposition \ref{proposition: properties of two-point estimator} (ii) and (iii) respectively, we know that 
\small
\begin{align*}
    Z_{\infty}\! :=\! \sup_{\substack{\hat{K} \in \mathcal{G}_{10J_0} \\ x_0,\delta}} \left\{\norm*{z_r(\hat{K}; x_0,\delta)}\right\} \leq d\lambda , \ \ Z_{2}\! := \!\sup_{\hat{K} \in \mathcal{G}_{10J_0}}\left\{\!\bbE\! \left[\norm*{z_r(\hat{K};x_0,\delta) - \bbE \left[z_r(\hat{K};x_0,\delta)  \mid \hat{K} \right] }^2\right]\right\} \leq d\lambda^2
\end{align*}
\normalsize
both hold, where we recall $\lambda >0$ is the (uniform) local Lipschitz parameter over $\mathcal{G}_{10J_0}$, $d = 2nk$ is the dimension of the optimization problem, and for notational convenience we denoted $\hat{K} = (K,\tilde{K})$. Plugging these upper bounds into our choice of $\eta$, it follows that it takes 
\begin{align*}
    T \geq \max\left\{\frac{960L d\lambda^2}{m\ep},  \frac{4L}{\mu}, \frac{4d\lambda}{\mu\rho}  \right\} \log(120\Delta_0/\ep)
\end{align*}
steps before we know that with probability at least 0.75, there exists $\Delta_T$ such that $\Delta_T \leq \ep.$
\end{proof}



\end{appendices}
\end{document}